\documentclass[11pt,reqno]{amsart}

\usepackage{booktabs}                                 
\usepackage[english]{babel}
\usepackage{latexsym}
 \usepackage[utf8]{inputenc}
\usepackage{babel}
\usepackage{fancyhdr}
\usepackage{longtable}
\usepackage{mathrsfs}
\usepackage{geometry}
\usepackage{comment}
\usepackage{textcomp}
\usepackage{caption}
\usepackage{lmodern}
\usepackage{mathrsfs}
 \usepackage[T1]{fontenc}
\usepackage[all,cmtip]{xy}
\usepackage{epsfig}
 \usepackage{amsmath,amsfonts,amssymb,amsthm}
\usepackage{graphics}
\usepackage{graphicx}
\usepackage{verbatim}
\usepackage{dsfont}
\usepackage{enumerate}  
\usepackage{pdflscape}
\usepackage{longtable}
\usepackage{booktabs}
\usepackage{colortbl}
\usepackage[shortlabels]{enumitem}
\usepackage[breaklinks]{hyperref} 
\hypersetup{
	colorlinks,
	linkcolor={red},
	citecolor={blue},
	urlcolor={red}
}
\usepackage{stmaryrd}
\usepackage{multirow}

\usepackage[disable]{todonotes}

\newcommand{\evnrow}{\rowcolor[gray]{0.95}}

\newcommand{\C}{\mathbb{C}}

\newcommand{\Q}{\mathbb{Q}}

\newcommand{\mQ}{\mathcal{Q}}

\newcommand{\mF}{\mathcal{F}}

\newcommand{\PP}{\mathbb{P}}

\newcommand{\mE}{\mathcal{E}}
\newcommand{\mU}{\mathcal{U}}

\newcommand{\tZ}{\widetilde{Z}}

\newcommand{\mZ}{\mathscr{Z}}
\newcommand{\of}{\mathcal{O}}

\newcommand{\W}{\bigwedge}

\newcommand{\Schur}{\Sigma}
\newcommand{\tY}{\widetilde{Y}}

\newcommand{\oU}{\overline{\mathcal{U}}}

\newcommand{\contr}{\lrcorner}

\DeclareMathOperator{\Aut}{Aut}

\DeclareMathOperator{\Ann}{Ann}

\DeclareMathOperator{\Ext}{Ext}

\DeclareMathOperator{\GL}{GL}

\DeclareMathOperator{\Sym}{Sym}
\DeclareMathOperator{\ddim}{dim}
\DeclareMathOperator{\Gr}{Gr}
\DeclareMathOperator{\Fl}{Fl}
\DeclareMathOperator{\OGr}{OGr}

\DeclareMathOperator{\codim}{codim}

\DeclareMathOperator{\Bl}{Bl}
\DeclareMathOperator{\SL}{SL}
\DeclareMathOperator{\SO}{SO}
\DeclareMathOperator{\rank}{rank}

\newtheorem{thm}{Theorem}[section]
\newtheorem{corollary}[thm]{Corollary}

\newtheorem{lemma}[thm]{Lemma}

\newtheorem*{aim*}{Aim of this paper}

\theoremstyle{definition}

\newtheorem{rmk}[thm]{Remark}

\newtheorem{caveat}[thm]{Caveat}

\makeatletter    
\def\l@subsection{\@tocline{1}{0,2pt}{2pc}{8mm}{\ \ }} 
\makeatletter    
\def\l@section{\@tocline{1}{0,2pt}{2pc}{8mm}{\ \ }} 

\author{Lorenzo De Biase}
\address{School of Mathematics - Cardiff University \\ CF24 4AG, UK}
\email[L.~De Biase]{debiase.lorenzo7@gmail.com}

\author{Enrico Fatighenti}
\address{ Dipartimento di Matematica\\
Sapienza Universit\`a di Roma\\
  Piazzale Aldo Moro 5\\ 00185 Roma, Italy.}
\email[E.~Fatighenti]{enricofatighenti6@gmail.com}

\author{Fabio Tanturri}
\address{Dipartimento di Matematica \\
Universit\`a di Genova\\
Via Dodecaneso 35\\
16146 Genova, Italy}
\email[F.~Tanturri]{tanturri@dima.unige.it}

\@namedef{subjclassname@2020}{\textup{2020} Mathematics Subject Classification}
 
\title[Fano 3-folds from homogeneous vector bundles over Grassmannians]{Fano 3-folds from homogeneous vector bundles over Grassmannians}

\begin{document}
\begin{abstract}
We rework the Mori--Mukai classification of Fano 3-folds, by describing each of the 105 families via biregular models as zero loci of general global sections of homogeneous vector bundles over products of Grassmannians.
\end{abstract}
\maketitle
\section{Introduction}

\thispagestyle{empty}

The classification of Fano 3-folds is one of the most influential results in birational geometry. Out of the 105 families, 17 have Picard rank $\rho=1$. They are usually called \emph{prime}. Their classification was completed first by Iskovskikh \cite{isk}, using the birational technique of the \emph{double projection from a line}. The classification was reworked by Mukai \cite{mukai}, using the biregular \emph{vector bundle method}. Mukai was able to describe most of the prime Fano varieties as complete intersections in certain homogeneous or quasi-homogeneous varieties. The latter in turn can be embedded in Grassmannians as zero loci of sections of homogeneous vector bundles.

Mori and Mukai \cite{morimukai} classified as well the 88 remaining families of Fano 3-folds with Picard rank $\rho \geq 2$. However, the proof has little in common with the vector bundle strategy, relying on the powerful birational Mori's theory of extremal rays.

One of the aims of this paper is to rewrite the entire classification of 3-folds in a biregular fashion, finding models for the non-prime Fano 3-folds which are akin to the Mukai's vector bundle ones.
In particular, for each of the 105 Fano $X$ we will look for a suitable embedding $X \subset \prod \Gr(k_i, n_i)$, such that $X$ can be described as the zero locus of a general global section of a homogeneous vector bundle $\mathcal{F}$ over $\prod \Gr(k_i, n_i)$.

In \cite{corti} Coates, Corti, Galkin, and Kasprzyk carried out a similar program. In particular, they were able to write down each of the 105 Fano 3-folds as zero loci of sections of vector bundles over GIT quotients. In some cases, their \emph{key varieties} are products of Grassmannians, and we decided to adopt their models. However, in many cases, their model of choice is a complete intersection in a toric variety, which was particularly suitable for their purpose of computing the quantum periods, with the aim of using ideas from mirror symmetry for further classification results.

Our motivating purpose is instead to attack the classification of Fano varieties in higher dimension from a representation-theoretical angle. In \cite{kuchle} K\"uchle classified Fano 4-folds of index 1 that can be obtained from completely reducible, homogeneous vector bundles over a single Grassmannian $\Gr(k,n)$. The resulting 20 families are therefore a sort of higher dimensional analogue of the Mukai models for 3-folds obtained via the vector bundle method. One of the main advantages of K\"uchle's method is that it relies only on very simple combinatorial data as input, such as the weight of the representation corresponding to the bundle involved. Moreover, this description allows an efficient computation of the invariants of the Fano, such as the Hodge numbers, for example using in combination the Koszul complex and Borel--Weil--Bott Theorem on the ambient variety. Such methods can be easily automatised via computer algebra, and extended to the case of products of Grassmannians $\prod \Gr(k_i, n_i)$, to say the least. This is exactly what we did. This paper originated from the construction of 3-folds via these methods; in a series of subsequent projects, we plan to work on more classification-type results, in dimension 4 and above.

As an initial benchmark for our strategy, we wanted to check how many of the 105 3-folds could be described using our methods. We found out that all 105 of them are. Although we do not believe that the same will be true in dimension 4 and higher, we hope to be able to find out many new and interesting examples of non-prime Fano 4-folds.

\subsection*{Main results}

The results of the paper are partially summarised in the following theorem. In what follows and throughout the whole paper, the notation $\mZ(\mF) \subset G$ will denote the zero locus of a general global section of the vector bundle $\mF$ in the variety $G$.
\begin{thm}\label{mainthm}
Let $X$ be a general smooth Fano 3-fold. Then there exist an ambient variety $G=\prod \Gr(k_i,n_i)$, product of (possibly weighted) Grassmannians, and a homogeneous vector bundle $\mathcal{F}$ on $G$ such that $X= \mZ(\mF) \subset G.$
\end{thm}
The only Fano varieties requiring weighted Grassmannians (actually, a unique weighted projective space) in their description without any alternative description are 1--11, 2--1, and 10--1, the others involving only classical Grassmannians. The weighted projective space in question is $\PP(1^3,2,3)$. The Fano 1--11 is a section of $\of(6)$ on the latter, 2--1 is a blow up of 1--11 and 10--1 is a linear section (multiplied with a $\PP^1$). Notice that for the Fano 1--11 (which was present in this form in Mukai's classification as well), $-K_X$ is not very ample. A few other weighted projective spaces appear, but for all of them we provide alternative descriptions.

In the statement of Theorem \ref{mainthm}, general means that a general Fano 3-fold $X$ in the corresponding family admits such a description. No hypothesis on the vector bundle $\mF$ is specified; our \emph{gold standard} for a homogeneous vector bundle $\mF$ is to be completely reducible and globally generated. Bundles with these properties are particularly suitable when facing classification problems. 
For 85 out of the 105 families, we managed to find a vector bundle of this form; for the remaining ones, we used homogeneous bundles which are extensions of some other homogeneous completely reducible ones, so that the description is slightly more complicated but still well within our range of techniques. Out of these 20 families, for 5 of them the vector bundle is particular: it
is of the form $\mF= \mF' \oplus \mathcal{G}$ where $\mathcal{G}$ is a line bundle with no global sections on the total space, but with sections on $\mZ(\mF')$. This happens when we need to blow up along a subvariety involving an exceptional divisor coming from a previous blow up. We deal with this phenomenon in Caveat \ref{caveatBundle}.

We partially collect these refinements in the following theorem.
\begin{thm} \label{thm:refined} Let $X$ be a Fano as in Theorem \ref{mainthm}. Then
\begin{itemize}
    \item For 102/105 families of Fano there exists a description without weighted factors in $G$.
    \item For 85/105 families of Fano there exists a description such that the bundle $\mF$ is completely reducible.
\end{itemize}
\end{thm}

The two theorems are proven in Section \ref{Fano3folds}, which we devote to the construction of the aforementioned families, except for those which are already known in the literature. We collect all the models in Section \ref{tables}; we include models for Del Pezzo surfaces as well. All models are general in moduli.

We draw the reader's attention to Section \ref{identifications} as well. This is mainly a collection of technical lemmas and results, and we believe that most of them are well-known to experts. 
Nonetheless, some of them are of independent interest, as they provide a dictionary between zero loci of sections of vector bundles and birational geometry. They were quite useful for translating Mori--Mukai models into our descriptions, and we believe that they can and will be useful for higher dimensional analyses. In this line of thought, we also present a few results involving flag varieties, even if they play only a small role in what follows.

\subsection*{Our models}

Mori--Mukai characterisation of the 88 non-prime 3-folds often involves intricate birational descriptions. The typical situation consists in blowing up a simpler 3-fold along a curve. Whenever the curve is a complete intersection in the base 3-fold, finding a suitable model in a product of Grassmannians is almost algorithmic; when the curve is not, then we perform a delicate analysis to understand how the curve can be cut in the ambient Fano. Subsequently, Lemma \ref{lem:blowup}, Corollary \ref{cor:cayleycrit}, and Lemma \ref{lem:blowDegeneracyLocus} allow us to describe the resulting 3-fold as a complete intersection in a suitable projective bundle. We then need to describe the latter as a zero locus of some vector bundle over a product of Grassmannians. In many cases, this is a straightforward procedure and the proof takes few lines. However, some projective bundles turn out to be particularly tricky, and we have to deal with them case-by-case. 

For other Fano we need to blow up a variety along a subvariety of codimension at least 3. To handle these cases, we collect and develop a few results which allow us to characterise these blow ups in term of zero loci of sections.

We want to give here an introductory example of a Fano 3-fold whose description is not immediate, yet admits a quite simple description in our model. We compare the original Mori--Mukai approach and the Coates--Corti--Galkin--Kasprzyk one with ours.

Let us consider the Fano of rank 2, number 16 in the Mori--Mukai list. Following the notation which will be adopted in our paper, we will call it 2--16.

\begin{description}[leftmargin=0pt]
\item[2--16, Mori--Mukai] Blow up of the complete intersection of two quadrics in $\PP^5$ in a conic $C$. Notice that $C$ is not a complete intersection in the ambient variety $\Q_1 \cap \Q_2 \subset \PP^5$.
\item[2--16, Coates--Corti--Galkin--Kasprzyk] A codimension-2 complete intersection
$\mZ(L+M, 2M) \subset F$ where $F$ has weight data
\[
\begin{array}{rrrrrrrl} 
	\multicolumn{1}{c}{s_0} & 
	\multicolumn{1}{c}{s_1} & 
	\multicolumn{1}{c}{s_2} & 
	\multicolumn{1}{c}{x} & 
	\multicolumn{1}{c}{x_3} & 
	\multicolumn{1}{c}{x_4} & 
	\multicolumn{1}{c}{x_5} & \\ 
	\cmidrule{1-7}
	1 & 1 & 1 & -1 & 0 & 0  & 0& \hspace{1.5ex} L\\ 
	0 & 0 & 0 & 1 & 1 & 1 & 1 & \hspace{1.5ex} M \\
\end{array}
\]
Equivalently, we can reformulate these data in terms of rays and cones. A way to do that is provided, e.g., in \cite{BelmansFatighentiTanturri}, and yields
{\small
\[
R = \{(-1,-1,-1,-1,-1), (0,0,0,0,1), (0,0,0,1,0), (0,0,1,0,0), (0,1,0,0,0), (1,0,0,0,0), (1,1,1,0,0)\},
\]
\vspace{-25pt}
\begin{multline*}
C = \{(1,2,3,5,6), (1,2,4,5,6), (0,2,3,5,6), (0,1,3,4,6), (0,1,2,3,5), (0,2,3,4,6), \\ (0,2,4,5,6), (0,1,3,5,6), (0,1,4,5,6), (0,1,2,4,5), (0,1,2,3,4), (1,2,3,4,6)\}.
\end{multline*}
}

\vspace{-12pt} \noindent The Fano variety is then the complete intersection of the two torus-invariant divisors $(0,0,0,0,0,2,1)$ and $(0,0,0,0,0,2,2)$, where for every divisor the $i$-th entry in the list corresponds to the coefficient of the $i$-th irreducible torus-invariant divisor.
\end{description}

Finally, our description realises this Fano as the zero locus of a general section of a globally generated homogeneous vector bundle over a (non-toric) product of Grassmannians.
\begin{description}[leftmargin=0pt]
\item[2--16, our description] The zero locus \[\mZ(\mU^{\vee}_{\Gr(2,4)}(1,0) \oplus \of(0,2)) \subset \PP^2 \times \Gr(2,4),\] where $\mU$ is the rank 2 tautological subbundle.
\end{description}

Our construction methods often allow for multiple models. For instance, the above Fano 2--16 can be realised as well as \[ \mZ(\of(1,0) \oplus \of(0,2)) \subset \Fl(1,2,4). \]
In order to preserve the compactness of this paper we usually decided to present only one model for each Fano variety, with notable exceptions whenever we found an alternative description too elegant not to include it, or whenever they were important intermediate steps in the identification of the model. Our choice of model depends on our personal taste. The criterion for an $X= \mZ(\mF) \subset \prod \Gr(k_i, n_i)$ was to pick the model with either the smallest number of factors or with the rank of $\mF$ as low as possible. To mention an example in lower dimension, the Del Pezzo surface of degree 5 can be equivalently described as $\mZ(\of(1,0,0,0,1)\oplus \of(0,1,0,0,1) \oplus \of(0,0,1,0,1) \oplus \of(0,0,0,1,1)) \subset (\PP^1)^4 \times \PP^2 $ or as $\mZ(\of(1)^{\oplus 4}) \subset \Gr(2,5)$. We will prefer the latter description to the former. 

\subsection*{Further directions}
\label{futureDirections}

As mentioned in the first part of the introduction, our methods are built with the explicit intention of being applied in higher dimension. Over a single Grassmannian $\Gr(k,n)$ homogeneous, completely reducible vector bundles can be written as direct sums of $\Sigma_{\alpha} \mQ \otimes \Sigma_{\beta}\mU$, where $\Sigma_{\alpha}$ (resp.\ $\Sigma_{\beta}$) denotes the Schur functor indexed by the non-increasing sequence $\alpha$ (resp.\ $\beta$); a similar expression holds for flag varieties and their products. This makes to some extent possible a methodical search for varieties which are zero loci of sections of bundles of this form.

What we plan to do in a series of subsequent works is to classify all Fano in dimension 4 that can be obtained in this way, comparing our results with the already existing known classes of Fano 4-folds (\cite{batyrev, coates, kalashnikov}, to cite a few). We are confident that many new and interesting examples can be found in this way, and the results of this paper are for sure strong motivations. We are particularly interested in the case of 4-folds of index 1 with Picard rank as high as possible and which are not a product. The \emph{champion} at the moment is the Fano of Picard rank 9 constructed by Casagrande, Codogni, and Fanelli in \cite{casagrande}; see, e.g., \cite{CasagrandeSurvey} for a survey of results on the topic.

Another case of interest are Fano varieties in higher dimension with special Hodge-theoretical properties. In particular, Fano varieties in any dimension of K3 type (in the sense of \cite{eg2}) have recently been studied due to their possible links with hyperk\"ahler manifolds. Finally, we remark that zero loci are particular cases of degeneracy loci of morphisms between vector bundles. It is certainly possible to further extend the above program to this framework, which has already been explored from many points of view (see, e.g., \cite{TanturriHilbert}), or even to the so-called orbital degeneracy loci, a recently introduced wider class of varieties \cite{BFMT2,BFMT}.

\subsection*{Plan of the paper} 
Section \ref{identifications} is where we establish our toolbox, and state or prove several lemmas, useful to translate the Mori--Mukai birational language into our biregular one, and vice versa. Section \ref{computeinvariants} is devoted to explaining how we are able to compute the invariants for all the models we present. Section \ref{Fano3folds} is the core of the paper. A detailed description of all the families which are not provided in the literature is given. Section \ref{tables} contains the tables and recap all the results in a schematic and handy fashion.

\subsection*{Notation and conventions}Throughout the whole paper, the notation $\mZ(\mF) \subset X$ denotes the zero locus of a general global section of the vector bundle $\mF$ in the variety $X$. We will denote by $X_d$ a general hypersurface of degree $d$ inside $X$.

If $E$ is a rank $r$ vector bundle over a variety $X$, we denote by $\PP_X(E)$ (or simply by $\PP(E)$ when no confusion can arise) the projective bundle $\pi:\mathrm{Proj}(\Sym E^{\vee}) \rightarrow X$; we remark that we adopt the subspace notation, as in \cite[Chapter 9]{EisenbudHarris3264}. If we denote by $\of_{\PP(E)}(1)$ (or simply $\of(1)$) the relatively ample line bundle, this yields $H^0(\PP(E),\of_{\PP(E)}(1)) \cong H^0(X, E^{\vee})$. Moreover $
\omega_{\PP(E)} \cong \of_{\PP(E)}(-r) \otimes \pi^*(\omega_X \otimes \det(E^{\vee}))
$
and, for any line bundle $L$, the isomorphism $\PP(E)\cong \PP(E\otimes L)$ induces $\of_{\PP (E)}(1) \otimes L^\vee = \of_{\PP(E\otimes L)}(1)$.

For products of varieties $X_1 \times X_2$, the expression $\mF_1 \boxtimes \mF_2$ will denote the tensor product between the pullbacks of $\mF_i$ via the natural projections. For products of Grassmannians $\Gr(k_1,n_1) \times \Gr(k_2,n_2)$, we will almost always adopt the short form $\of(a,b):=\of(a) \boxtimes \of(b)$; we will often omit the pullbacks when no confusion can arise, so that, e.g., $\mQ_{\Gr(k_1,n_1)}(1,2)=\mQ_{\Gr(k_1,n_1)}(1) \boxtimes \of_{\Gr(k_2,n_2)}(2)$. 

By $\Fl(k_1, \ldots, k_r, n)$ we will denote the flag variety of subspaces $V_{k_1} \subset V_{k_2} \subset \ldots \subset V_{k_r} \subset \C^n$. We will denote by $\pi_i$ the projection to the $i$-th Grassmannian $\Gr(k_i,n)$. $\mU_i$ and $\mQ_i$ will denote the pullback of the tautological subbundle and quotient bundle via $\pi_i$, of rank $k_i, n-k_i$ respectively. For short, we will write $\of(a,b)=\pi_1^*(\of(a)) \otimes \pi_2^*(\of(b))$. In the rare cases where a flag is involved in a product of varieties, the different Picard groups will be separated by a semicolon, i.e., $\of(a,b;c)=\of(a,b) \boxtimes \of(c)$ on $\Fl(k_1,k_2,n) \times \Gr(k',n')$.

Many data for Table \ref{tab:3folds} (and for the paper overall) are taken from \cite{fanography}. They rely on the tables from \cite{isp5,pcs,kps,cfst}. Many other alternative descriptions are taken from \cite{corti}. We include the relevant citation to the alternative description in the table whenever appropriate. The notation X--Y for a Fano means a Fano of Picard rank $X$ which is the number $Y$ in the Mori--Mukai list.
Finally, $\mathbb{Q}_3$ denotes the 3-dimensional quadric hypersurface (Fano 1--16) and $\mathbb{V}_5$ denotes the index 2, degree 5 linear section of $\Gr(2,5)$ (Fano 1--15).
\subsection*{Acknowledgements} 
We are indebted to Daniele Faenzi for many enlightening suggestions. Thanks to Vladimiro Benedetti, Marcello Bernardara, Giovanni Mongardi, and Miles Reid for useful discussions. We also thank the referees for their useful comments. EF and FT were partially supported by a ``Research in Paris'' grant held at Institut Henri Poincar\'e. We thank the institute for the excellent working conditions. We acknowledge the Laboratoire Paul Painlev\'e -- Universit\'e de Lille, the Dipartimento di Matematica ``Giuseppe Peano'' -- Universit\`a di Torino and INdAM for partial support as well. All three authors are members of INdAM-GNSAGA.

\section{Identifications}
\label{identifications}

Most of the Fano 3-folds with Picard rank $\rho\geq 2$ arise as blow up of other Fano 3-folds with centre in distinguished subvarieties. Sometimes other standard birational descriptions are involved. The purpose of this subsection is therefore to establish a toolbox that allow us to translate the Mori--Mukai birational language into models suitable for our type of descriptions. Most of the lemmas appearing in this section are probably well-known to the experts: however for some of them we have not been able to locate clear proofs in the literature. 

The most basic result is the description of the blow up of a projective space in a linear subspace. We will use the following lemma:


\begin{lemma} \label{lem:blow}
Let $\mQ$ be the tautological quotient bundle on $ \PP^{n-r}$. We have
\[
\Bl_{\PP^{r-1}}\PP^n \cong \mZ(\mQ_{\PP^{n-r}}(0,1)) \subset \PP^{n-r} \times \PP^n.
\]

\begin{proof}
Let $V$ be a $n+1$-dimensional vector space such that $\PP^n \cong \PP(V)$. By \cite[Proposition 9.11]{EisenbudHarris3264},
$\Bl_{\PP^{r-1}}\PP^n$ is isomorphic to the projectivization of the vector bundle
\[
E = \of_{\PP^{n-r}}(-1) \oplus (V' \otimes \of_{\PP^{n-r}}),
\]
where $V' \subset V$ has dimension $r$ and $\PP^{n-r}$ is identified with $\PP(V/V')$. The bundle $E$ fits into the short exact sequence
\[
0 \rightarrow E \rightarrow (V/V' \oplus V') \otimes \of_{\PP^{n-r}} \rightarrow \mQ_{\PP^{n-r}} \rightarrow 0,
\]
hence $\PP(E)$ can be also expressed as the zero locus of $\mQ \boxtimes \of(1)$ inside $\PP^{n-r} \times \PP(V/V' \oplus V') \cong \PP^{n-r} \times \PP^n$, as claimed.
\end{proof}
\end{lemma}

In the above lemma we used the fact that, as soon as we have a short exact sequence on $X$ of vector bundles
$0 \rightarrow
E \rightarrow
F \rightarrow
G \rightarrow
0,$
then $\PP(E)$ can be obtained as the zero locus of a section of $t$ of $\pi^*(G) \otimes \of_{\PP(F)}(1)$ over $\pi:\PP(F)\rightarrow X$. If $H^1(E)=0$, then $t$ can be chosen to be general; a particularly favourable situation will occur when $F\cong \of^{\oplus r}$, so that $\PP(E)$ embeds into $X \times \PP^{r-1}$.

Lemma \ref{lem:blow} can be generalised for the Grassmannians context.
\begin{lemma}\label{lem:blowInGrass}
We have
\[
\Bl_{\Gr(k-1, n-1)}\Gr(k,n) \cong \mZ(\mQ \boxtimes \mU^{\vee}) \subset \Gr(k,n-1) \times \Gr(k,n),
\]
where the centre of the blow up $\Gr(k-1, n-1)$ is identified with $\mZ(\mQ)\subset \Gr(k, n)$.
\begin{proof}
Let $V_n, V_{n-1}$ be complex vector spaces of dimension $n, n-1$ respectively. A section of $\mQ \boxtimes \mU^{\vee}$ over $\Gr(k,V_{n-1}) \times \Gr(k,V_n)$ can be regarded as a section of $\mU^{\vee} \boxtimes \mU^{\vee}$ over $\Gr(n-k-1,V_{n-1}^\vee) \times \Gr(k,V_n)$, and the corresponding zero loci are canonically isomorphic.

A section of the latter vector bundle is of the form
\[
s=f_1 x_1 + \ldots + f_{n-1} x_{n-1}
\]
for some $f_i \in V_{n-1}$, $x_i \in V_n^\vee$. Let us fix bases for $V_{n-1}^{\vee}, V_n$ accordingly. Up to the action of $\GL_k$, a point in $\Gr(k,V_n)$ is represented by a $k \times n$ matrix
\[
A=\left( \begin{array}{ccc}
a_{1,1} & \cdots & a_{1,n}\\
\vdots & & \vdots \\
a_{k,1} & \cdots & a_{k,n}
\end{array}
\right)
\]
and a section $x_i \in V_n^\vee$ evaluates as $x_i (A)=(a_{1,i} \cdots a_{k,i})^t$. Analogously, a section $f_i \in V_n$ evaluates on a point $B \in \Gr(n-k-1,V_{n-1}^\vee)$, seen as a $n-k-1 \times n-1$ matrix, as $f_i(B)=(b_{1,i} \cdots b_{n-k-1,i})^t$.

The evaluation of a section $f_i x_i$ on a point $(A,B)$ is given by the $k \times n-k-1$ matrix $x_i(A) \cdot (f_i(B))^t$. It is straightforward to check that
\[
s(A,B) = 0 \qquad \mbox{if and only if} \qquad A \cdot \left( 
\begin{array}{c}
B^t \\
\hline 0 \cdots 0
\end{array} \right)=0.
\]

Let $Y$ be the zero locus of $s$. We want to study the fibres of the (restriction of the) natural projection $Y \rightarrow \Gr(k,V_n)$. This amounts to solving a linear system with $b_{1,1}, b_{1,2}, \ldots, b_{1,n-1}, b_{2,1}, \ldots, b_{n-k-1,n-1}$ as variables. With this choice of coordinates, the matrix associated to the linear system is the $(n-1)(n-k-1) \times k(n-k-1)$ matrix
\[
\left(
\begin{array}{cccc}
\tilde{A} \\
& \tilde{A} \\
& & \ddots \\
& & & \tilde{A}
\end{array}
\right), \qquad \mbox{where } A=\left(
\begin{array}{c|c}
\tilde{A} & \begin{array}{c}
a_{1,n}\\
\vdots\\
a_{k,n}
\end{array}
\end{array}
\right).
\]
The fiber over a general point $A$, i.e., whenever $\tilde{A}$ has maximal rank, is a single point $\in \Gr(n-k-1,V_{n-1}^\vee)$, hence $Y \rightarrow \Gr(k,V_n)$ is birational. The fiber over $A$ is positive-dimensional if and only if $\tilde{A}$ has rank at most $k-1$, i.e., if and only if
\[
\rank \left(
\begin{array}{c}
A \\
\hline 0 \cdots 0 \, 1
\end{array}
\right)<k+1.
\]
Up to choosing the last element of a basis of $V_n$, we may assume that a general section of $\mQ$ over $\Gr(k,V_n)$ is $x_n$, hence the locus in $\Gr(k,V_n)$ where the map is not birational is precisely $\mZ(\mQ) \cong \Gr(k-1, n-1)$. Further degenerations would occur over points $A$ with $\tilde{A}$ having rank at most $k-2$, which is not possible for a point in $\Gr(k, V_n)$.
\end{proof}
\end{lemma}

Geometrically, the proof of the previous lemma can be interpreted in terms of the graph of the rational map $\xymatrix{\Gr(k, V_n) \ar@{-->}[r] & \Gr(k, V_{n-1})}$ induced by the projection from the point $x_n^\vee$.

Sometimes it is convenient to interpret blow ups of projective spaces in linear subspaces as projective bundles. The following remark is very useful in this context.
\begin{rmk}[{\cite{wisniewski1991fano}}] \label{rmk:wisniewski}
The following identifications hold.
\begin{itemize}
    \item Let $n$ be an odd number. $\PP_{\PP^{\frac{n+1}{2}}}(T_{\PP^{\frac{n+1}{2}}}) \cong X_{(1,1)} \subset \PP^{\frac{n+1}{2}} \times \PP^{\frac{n+1}{2}}.$
    \item $\Bl_{\PP^{r-1}}\PP^{n} \cong \PP_{\PP^{n-r}}(  \of(-1) \oplus \of^{\oplus r}).$
\end{itemize}

\end{rmk}
 
For similar projective bundles we might not have a neat description as blow ups. However these projective bundles appear frequently, especially in the toric cases. It is then convenient to describe them as zero loci of vector bundles. 
\begin{lemma} \label{lem:pxp} On $\PP^m \times \PP^n$, for any $k,h \in \mathbb{N}$ such that $n=k(m+1)+h$ we have
\[
\mZ(\mQ_{\PP^m}(0,1)^{\oplus k}) \cong \PP(\of_{\PP^m}^{\oplus k}(-1) \oplus \of_{\PP^m}^{\oplus h+1}).
\]
\begin{proof}
It easily follows from the short exact sequence on $\PP^m$
\[
0 \rightarrow
\of(-1)^{\oplus k} \oplus \of^{\oplus h+1} \rightarrow
\of^{\oplus n+1} \rightarrow
\mQ^{\oplus k} \rightarrow
0,
\]
which is obtained by adding $k$ Euler sequences and the trivial sequence $\of^{\oplus h+1} \rightarrow \of^{\oplus h+1}$.
\end{proof}
\end{lemma}

Finally, we tackle the case of flag varieties. Sometimes our models for some Fano 3-folds are easily identified as sections of very simple bundles (e.g., linear sections) on flag varieties. It is important to be able to identify subvarieties of flag varieties with appropriate subvarieties in the product of Grassmannians. As a first step we prove the following lemma.

\begin{lemma}\label{lem:identificationsOnFlags} Let $\Fl(k_1, k_2, n)$ be a two-step flag. We have the following identifications:
\begin{gather*} \mZ(\mQ_2) \subset \Fl(k_1, k_2, n)\cong \Gr_{\Gr(k_2-1, n-1)}(k_1, \overline{\mU} \oplus \of); \\  \mZ(\mU_1^{\vee}) \subset \Fl(k_1, k_2, n)\cong \Gr_{\Gr(k_1, n-1)}(k_2-k_1, \overline{\mQ}(-1) \oplus \of(-1)), 
\end{gather*}
where $\overline{\mU}$ (resp., $\overline{\mQ}$) denotes the tautological subbundle on $\Gr(k_2-1, n-1)$ (resp., the tautological quotient bundle on $\Gr(k_1, n-1)$) and $\Gr(k,E)$ denotes the Grassmann bundle and $\mQ_2$ denotes the pullback of the quotient bundle on $\Gr(k_2, n)$ to the flag $\Fl(k_1, k_2, n)$.
\end{lemma}
\begin{proof}
Recall that we can interpret the two-step flag $\Fl(k_1, k_2,n)$  as 
\[
\Fl(k_1, k_2,n) \cong\Gr_{\Gr(k_2, n)}(k_1, \mU) \cong \Gr_{\Gr(k_1, n)}(k_2-k_1, \mQ(-1)).
\]
The zero locus $\mZ(\mQ) \subset \Gr(k_2,n)$ is isomorphic to $\Gr(k_2-1, n-1)$.
Under this isomorphism we have 
\[
\mU|_{\mZ(\mQ)} \cong  \overline{\mU} \oplus \of.
\]

Similarly the zero locus $\mZ(\mU^{\vee}) \subset \Gr(k_1,n)$ is isomorphic to $\Gr(k_1, n-1)$. Under this isomorphism we have \[ \mQ(-1)|_{\mZ(\mU^{\vee})} \cong \overline{\mQ}(-1) \oplus \of(-1).\]
The result then follows.
\end{proof}
The above lemma has a particularly simple formulation on $\Fl(1,2,n)$, which is worth to make explicit for future references. Similar results can be obtained on $\Fl(1,k,n)$ and $\Fl(k-1,k,n)$ using an appropriate number of copies of $\mQ_2$ and $\mU_1^{\vee}.$

\begin{corollary} \label{cor:onF12n}
On $\Fl(1,2,n)$ one has
\begin{gather*} \mZ(\mQ_2) \subset \Fl(1, 2, n))\cong \PP_{\PP^{n-2}}(\of (-1) \oplus \of); \\
\mZ(\mU_1^{\vee}) \subset \Fl(1,2, n))\cong \PP_{\PP^{n-2}}({\mQ}(-1) \oplus \of(-1) ). 
\end{gather*}
\end{corollary}

Putting together the above results we get the following.
\begin{corollary} \label{cor:blowupflag}
One has the following isomorphisms:
\[\Bl_{\PP^{r-1}}\PP^n \cong \mZ(\mQ_2^{\oplus r}) \subset \Fl(1, r+1, n+1) \cong \mZ(\mQ_{\PP^{n-r}}(0,1)) \subset \PP^{n-r} \times \PP^n.\]

\end{corollary}

In many occasions we will need to blow up along non-linear subvarieties. The following lemma gives an easy description that applies to the case of complete intersection curves cut by divisors of the same (multi)degree. 
\begin{lemma}[{\cite[Lemma 2.2]{eg2}}]
\label{lem:blowup} Let $X:=X_(d,1)$ be a general hypersurface of bidegree $(d,1)$ in $Z \times \PP^1$. Then $X \cong Bl_S Z,$ where $S$ is the intersections of $2$ divisors of degree $d$ on $ Z$.
\end{lemma}
The lemma above will turn out to be handy in a number of circumstances. Although it is stated here for $Z$ prime, it admits an obvious generalisation when the Picard rank of $Z$ is greater than 1.

Lemma \ref{lem:blowup} admits a classical generalisation in higher codimension, known as the \emph{Cayley trick} (see \cite[Thm 2.4]{kimkim}, or \cite[3.7]{ilievmanivel}), which in turn can be considered as a generalisation of Orlov's formula for the derived category of blow ups. The setup is the following: assume that we have $Y= \mZ(A) \subset S$, where $A$ is ample of rank $r\geq 2$. We have a natural isomorphism $H^0(S,A) \cong H^0(\PP(A^{\vee}), \of_{\PP(A^{\vee})}(1))$. The same section defining $Y$ defines a hypersurface $X$ in $\PP(A^{\vee})$. The complete result is that the derived category of $X$ admits a semiorthogonal decomposition containing $r-1$ copies of $D^b(Y)$. When the rank of $A$ is exactly 2, this produces a generalisation of the above Lemma \ref{lem:blowup}.
\begin{lemma}[{\cite[Lemma 3.2]{kimkim}}] \label{lem:cayley} Let $S, X, Y$ be as above, and let $A$ be ample of rank 2. Then $X \cong \Bl_Y S$.
\end{lemma}

How can we effectively use the Cayley trick in our case? Assume that we have a bundle $F=E \oplus G$ on $\Fl(1,2,n)$ with $G= \pi_2^* \widetilde{G}$ for a bundle $\widetilde{G}$ on $\Gr(2,n)$. Take $\sigma \oplus \mu \in H^0(\Fl(1,2,n), E \oplus G)$, which defines the chain of inclusions $X:=V(\sigma \oplus \mu) \subset Y:=V(\mu)\subset \Fl(1,2,n)$. The section $\mu$ induces a section $(\pi_2)_*(\mu) \in H^0(\Gr(2,n),\widetilde{G})$. Let $\widetilde{Y}:=V((\pi_2)_* \mu) \subset \Gr(2,n)$; in particular, $X$ can be regarded as $V(\sigma) \subset \PP_{\widetilde{Y}}(\mU|_{\widetilde{Y}})$.

Suppose there is an identification $\varphi:H^0(Y, E|_Y) \rightarrow H^0(\widetilde{Y},\mU|_{\widetilde{Y}} \otimes L)$, where $L$ is a line bundle on $\widetilde{Y}$ (note that for any $\mE, L$, $\PP(\mE) \cong \PP(\mE \otimes L)$). We are thus in the following situation:

\[
\xymatrix@C=5pt{
\Fl(1,2,n) \ar[d]^-{\pi_2}& \supset & Y:=V(\mu) & \supset & X:=V(\sigma \oplus \mu) \\
\Gr(2,n) & \supset & \widetilde{Y}:=V((\pi_2)_* \mu) & \supset & \widetilde{X}:=V(\varphi(\sigma|_Y))
}
\]

Then by Lemma \ref{lem:cayley} we have
\[
X \cong \Bl_{\widetilde{X}} \widetilde{Y}.
\]

The ampleness of $\mU|_{\widetilde{Y}} \otimes L$ required in Lemma \ref{lem:cayley} is not necessary, as shown in \cite[Proposition 46]{bfm}. We remark that $\sigma|_Y$ is general in $H^0(Y, E|_Y)$ if $H^i(\Fl,\W^i G^\vee \otimes E)$ vanishes for any $i>0$, which implies that $H^0(\Fl, E)$ surjects onto $H^0(Y, E|_Y)$.


We do a recap in the following handy corollary. This is will be useful to deal with the case of complete intersection curves of different degrees.
\begin{corollary}\label{cor:cayleycrit}
Assume that we have a bundle $F=E \oplus G$ on $\Fl(1,2,n)$, with $G=\pi_2^*\widetilde{G}$ for a bundle $\widetilde{G}$ on $\Gr(2,n)$ and  $X=\mZ(F) \subset  Y= \mZ(G) \subset \Fl(1,2,n)$. Denote by $\widetilde{Y}$ the zero locus $\widetilde{Y}=\mZ(\widetilde{G}) \subset \Gr(2,n)$. Assume that
$H^0(Y, E|_Y) \cong H^0(\widetilde{Y},\mU|_{\widetilde{Y}} \otimes L)$ for some line bundle $L$ on $\widetilde{Y}$. Denote by $\widetilde{X}=\mZ(\mU|_{\widetilde{Y}} \otimes L)\subset \tY$. Then $X \cong \Bl_{\widetilde{X}} \widetilde{Y}$.
\end{corollary}

There is a further generalisation of the Cayley trick, that applies to degeneracy loci as well, which we recall for completeness.

\begin{lemma}[{\cite[Lemma 2.1]{kuznetsovKuchle}}]
\label{lem:blowDegeneracyLocus}
Let $\varphi:E \rightarrow F$ be a morphism of vector bundles of ranks $r+1$, $r$ respectively on a Cohen--Macaulay variety $X$.
Denote by $D_k(\varphi)$ the $k$-th degeneracy locus of $\varphi$, i.e., the locus where the morphism has corank at least $k$.
Consider the projectivization $\pi:\PP(E) \rightarrow X$, then $\varphi$ gives a global section of the vector bundle $\pi^*F \otimes \of(1)$.
If $\codim D_k(\varphi) \ge k + 1$ for all $k \ge 1$ then the zero locus of $\varphi$ on $\PP(E)$ is isomorphic to the blow up of $X$ along $D_1(\varphi)$.
\end{lemma}

In practice, we will often need to find some projective bundle $\PP(\of(-d_1,\dotsc,-d_m) \oplus \of^{\oplus r})$ as the zero locus of a suitable vector bundle over a product of Grassmannians. The following remark will be very helpful for this sake; an instance of its application will be Lemma \ref{projBundle1-12}. 

\begin{rmk}
\label{rem:principalParts} Let $L$ be a line bundle on $X$. For any $k$ one can define $\mathcal{P}^k(L)$, the bundle of $k$-principal parts of $L$, of rank $\binom{k+\ddim X}{k}$. One has $\mathcal{P}^0(L)=L$; by \cite[Exp II, Appendix II 1.2.4.]{SGA} there exist natural short exact sequences
\begin{equation}\label{seq:principalparts} 0 \rightarrow \Sym^k (\Omega_X)(L) \rightarrow \mathcal{P}^k(L) \rightarrow \mathcal{P}^{k-1}(L) \rightarrow 0. \end{equation}
If $X \cong \PP^n$ these bundles of principal parts are homogeneous, and in \cite[Thm 1.1]{re} their splitting type is determined. The situation is particularly simple when we consider $L=\of_{\PP^n}(d)$ with $d \geq k$: in this case one has $\mathcal{P}^k(\of_{\PP^n}(d)) \cong \Sym^k V_{n+1} \otimes \of_{\PP^n}(d-k)$. The sequence above for $k=1$ coincides with the dualised twisted Euler sequence.
\end{rmk}

We finish this section with a classical remark on how to characterise double covers as hypersurfaces in projective bundles. A detailed proof can be found for example in \cite[Lemma 1.2]{lyu}. The formula below can be easily generalised to the case of $k$-cyclic covers, using $\of_P(k)$ instead.

\begin{rmk}\label{lem:doublecovers}
Let $X$ be a 2-fold cyclic covering of $Y$, ramified along a smooth divisor $D$, and $L$ a line bundle with $L^{\otimes 2}=\of_Y(D)$, which is assumed to be $2$-divisible in $\mathrm{Pic}(Y)$. Then $X$ can be identified with $\mZ(\of_P(2))$ in $P:=\PP_Y(\of \oplus L^{\vee})$. 
\end{rmk}

\section{How to compute invariants}
\label{computeinvariants}
In this section we explain and show with a concrete example how we can compute the invariants of a zero locus of a general section of a given homogeneous vector bundle on a product of flag varieties.

As a matter of fact, such computations are not strictly necessary for the identification of the models we found for the Fano 3-folds in the next section. However, we want to stress out the importance of having such a tool for two reasons. On the one hand, one could start producing in an automatised way many examples coming from homogeneous vector bundles on products of flag varieties and later try to identify them using the existing classifications. This was exactly the starting point of this project and what made us able to characterise, along the process, many zero loci of sections from a geometric point of view. Several results of Section \ref{identifications} have been found by trying to generalise the evidences coming from all the examples we had. On the other hand, it goes without saying that these methods will certainly be very useful when a similar search will be performed for varieties which have not yet been classified.

\subsection{The invariants \texorpdfstring{$h^0(-K)$ and $(-K)^3$}{h\^{}0(-K) and (-K)\^{}3}}

These invariants can be computed via intersection theory. If $X$ is a product of flag varieties, then we know its graded intersection ring of algebraic cycle classes modulo numerical equivalence. We know how to integrate, so that Hirzebruch--Riemann--Roch Theorem yields a way to compute $\chi(E)$ for any vector bundle $E$ with assigned Chern classes.

The situation does not change much when we consider a subvariety $\mZ(\mF) \subset X$ given as the zero locus of a general section of some vector bundle $\mF$ on $X$. If we know the Chern classes of $\mF$, we can write down the graded intersection ring of $\mZ(\mF)$, as well as count points on $0$-dimensional cycles.

In concrete examples, instead of doing computations by hand, it is of course convenient to use some computer algebra software. Our choice fell on \cite{M2}, for which an already developed package \cite{Schubert2} implementing the methods we need is available. This allows us to compute $(-K_{\mZ(\mF)})^3$, as the Chern classes of the canonical sheaf of $\mZ(\mF)$ are easy to express. As for $h^0(-K_{\mZ(\mF)})$, we certainly know how to compute $\chi(-K_{\mZ(\mF)})$. But $-K=K-2K$ and $-2K$ is ample, so the Kodaira Vanishing Theorem implies $h^{i>0}(-K_{\mZ(\mF)})=0$.

\subsection{Hodge numbers and tangent cohomology}

Beside the aforementioned invariants, one of the most important data one would like to know about a Fano variety is its Picard rank. More in general, it is rather important to compute $h^{i,j}$ for a given variety. In our setting this is perfectly doable using classical tools as the Koszul complex and a bit of representation theory, even though the computations may quickly become cumbersome if the involved vector bundles have high rank or several summands.

We briefly recall the strategy. Let us suppose that $Y=\mZ(\mF)\subset X$. Assume that $\rank(\mF)=r$. For each $j \in \mathbb{N}$, we have the $j$-th exterior power of the conormal sequence
\begin{equation}
\label{wedgeKConormal}
0\rightarrow
\Sym^j \mF^\vee|_Y \rightarrow
(\Sym^{j-1} \mF^\vee \otimes \Omega_X)|_Y \rightarrow
\dotso \rightarrow
(\Sym^{j-k} \mF^\vee \otimes \Omega^k_X)|_Y \rightarrow
\dotso \rightarrow
\Omega^j_X|_Y \rightarrow
\Omega^j_Y \rightarrow
0.
\end{equation}
As our goal is $h^i(\Omega^j_Y)$, we can compute the dimensions of the cohomology groups of all the other terms in \eqref{wedgeKConormal}, split it into short exact sequences and use the induced long exact sequences in cohomology to get the result.

Each term $(\Sym^{j-k} \mF^\vee \otimes \Omega^k_X)|_Y$ is in turn resolved by an exact Koszul complex
\begin{equation*}
0\rightarrow
\W^r \mF^\vee \otimes \Sym^{j-k} \mF^\vee \otimes \Omega^k_X \rightarrow
\dotso \rightarrow
\mF^\vee \otimes \Sym^{j-k} \mF^\vee \otimes \Omega^k_X \rightarrow
\Sym^{j-k} \mF^\vee \otimes \Omega^k_X,
\end{equation*}
so that we are led to compute the cohomology groups of the terms above. If $X$ is a product of Grassmannians and $F$ is completely reducible, then those terms are completely reducible as well: a decomposition can be found via suitable plethysms. The cohomology groups can be then obtained via the usual Borel--Weil--Bott Theorem. 

Things get worse if $X$ has some genuine flag variety as a factor, in which case $\Omega_X$ is an extension of completely reducible vector bundles, or if $F$ itself is an extension thereof. In these cases, one needs to deal carefully with the exterior/symmetric power of an extension (which is an extension itself) and the tensor product of extensions; in the end, each term of the Koszul complex above is again an extension of completely reducible vector bundles, whose cohomology groups can be easily computed and arranged to get the result.

It may happen that several cohomology groups do not vanish, so that in the induced long exact sequences in cohomology there are boundary homomorphisms whose rank is a priori not known. This leads to some ambiguity in the final results, and can be partially solved by considering the additional relations involving $h^{i,j}$ such as the symmetries in the Hodge diamond and the computation of $\chi(\Omega^j_Y)$ as done above.

Additionally, suppose that we want to get some information on the automorphism group and the space of deformations of $Y$. One way is to compute $h^0(T_Y)$ and $h^1(T_Y)$ via the normal sequence
\[
0 \rightarrow
T_Y \rightarrow
T_X|_Y \rightarrow
\mF|_Y \rightarrow
0.
\]
As before, one can compute the cohomology groups of the terms on the right via the usual Koszul complex and get some information on $h^i(T_Y)$.

A rather easy example of application of the whole routine is provided in Section \ref{aworkedexample}. It is evident that such computations cannot be done by hand for more complicated examples, especially for a significant number of cases. A Macaulay2 \cite{M2} package which was developed to implement and automatise the procedure just described will be presented in \cite{FatighentiTanturriPackage}.

\subsection{A worked example}
\label{aworkedexample}

Let us show how to concretely compute the Hodge numbers of $Y:= \mZ(\mU^{\vee}_{\Gr(2,4)}(1,0) \oplus \of(0,2)) \subset \PP^2 \times \Gr(2,4)=:X$, which we will prove to be a model for \hyperlink{Fano2--16}{2--16}.

Our vector bundle $\mF:=\mU^{\vee}_{\Gr(2,4)}(1,0) \oplus \of(0,2)$ has rank $3$. For $j=0$, \eqref{wedgeKConormal} simply becomes $\of_Y \rightarrow \of_Y$. The Koszul complex resolving $\of_Y$ is
\[
0 \rightarrow
\of(-2,-3) \rightarrow
\mU_{\Gr(2,4)}(-1,-2) \oplus \of(-2,-1) \rightarrow
\of(0,-2) \oplus \mU_{\Gr(2,4)}(-1,0) \rightarrow
\of;
\]
the only non-zero cohomology group is $H^0(\of)\cong \mathbb{C}$, which gives $h^{0,0}=1$ and $h^{0,j}=0$ for $j>0$.

For $j=1$, \eqref{wedgeKConormal} yields the usual conormal short exact sequence. The term on the left is $\mF^\vee|_Y$, which is resolved by a Koszul complex whose terms are
\begin{equation*}
\begin{array}{rcl}
\W^3 \mF^\vee \otimes \mF^\vee & = &\of(-2,-5) \oplus \mU_{\Gr(2,4)}(-3,-3) \\
\W^2 \mF^\vee \otimes \mF^\vee & = & \mU_{\Gr(2,4)}(-1,-4) \oplus \Sym^2\mU_{\Gr(2,4)}(-2,-2)\oplus \of(-2,-3)^{\oplus 2} \oplus\mU_{\Gr(2,4)}(-3,-1) \\
\mF^\vee \otimes \mF^\vee & = & \of(0,-4) \oplus 
\mU_{\Gr(2,4)}(-1,-2)^{\oplus 2} \oplus 
\Sym^2 \mU_{\Gr(2,4)}(-2,0) \oplus 
\of(-2,-1)\\
\mF^\vee & = & \of(0,-2) \oplus 
\mU_{\Gr(2,4)}(-1,0);
\end{array}
\end{equation*}
the only non-zero cohomology group is $h^4(\mF^\vee \otimes \mF^\vee)=1$, which yields $h^3(\mF^\vee|_Y)=1$.

The middle term is $\Omega_X|_Y$, which is resolved by a Koszul complex whose terms are
\begin{equation*}
\begin{array}{rcl}
\W^3 \mF^\vee \otimes \Omega_X &= & 
\mQ_{\Gr(2,4)} \otimes \mU_{\Gr(2,4)}(-2,-4) \oplus \mQ_{\PP^2}(-4,-3)
\\
\W^2 \mF^\vee \otimes \Omega_X & = & 
%
\mQ_{\Gr(2,4)}\otimes\left(\Sym^2\mU_{\Gr(2,4)}(-1,-3)
\oplus \of(-1,-4) \oplus \mU_{\Gr(2,4)}(-2,-2)
\right) \oplus \\ & & {}\oplus
 \mQ_{\PP^2}(-4,-1) \oplus \mU_{\Gr(2,4)} \otimes \mQ_{\PP^2}(-3,-2)
%
\\
\mF^\vee \otimes \Omega_X & = & 
\mQ_{\Gr(2,4)}\otimes\left(\mU_{\Gr(2,4)}(0,-3)
\oplus \Sym^2\mU_{\Gr(2,4)}(-1,-1)
\oplus \of(-1,-2)
\right) \oplus \\ & & {}\oplus
 \mQ_{\PP^2}(-2,-2) \oplus \mU_{\Gr(2,4)}\otimes \mQ_{\PP^2}(-3,0)\\
\Omega_X & = & \mQ_{\Gr(2,4)}\otimes \mU_{\Gr(2,4)}(0,-1) \oplus \mQ_{\PP^2}(-2,0);
\end{array}
\end{equation*}
the only non-zero cohomology groups are $h^3(\mF^\vee \otimes \Omega_X)=1$ and $h^1(\Omega_X)=2$, which yield $h^1(\Omega_X|_Y)=2$ and $h^2(\Omega_X|_Y)=1$.

The long exact sequence in cohomology induced by the conormal sequence then gives $h^{1,1}=2$ and $h^{1,2}=2$, while the other $h^{1,j}$ are zero.

Similar computations can be performed to compute $h^i(T_Y)$, by considering the normal sequence. The middle term is $T_X|_Y$, which is resolved by a Koszul complex whose terms are
\begin{equation*}
\begin{array}{rcl}
\W^3 \mF^\vee \otimes T_X &= & 
\mQ_{\Gr(2,4)} \otimes \mU_{\Gr(2,4)}(-2,-2) \oplus \mQ_{\PP^2}(-1,-3) 
\\
\W^2 \mF^\vee \otimes T_X & = & 
\mQ_{\Gr(2,4)} \otimes \left( 
\Sym^2 \mU_{\Gr(2,4)}(-1,-1) \oplus \of(-1,-2) \oplus \mU_{\Gr(2,4)}(-2,0)
\right) \oplus \\ & & {}\oplus
\mQ_{\PP^2} \otimes \left( 
\mU_{\Gr(2,4)}(0,-2) \oplus \of(-1,-1) 
\right)
\\
\mF^\vee \otimes T_X & = & 
\mQ_{\Gr(2,4)} \otimes \left(
\mU_{\Gr(2,4)}(0,-1) \oplus
\Sym^2 \mU_{\Gr(2,4)}(-1,1) \oplus
\of(-1,0)
\right) \oplus \\ & & {}\oplus
\mQ_{\PP^2} \otimes \left( 
\of(1,-2) \oplus \mU_{\Gr(2,4)}
\right)
\\
T_X & = & 
\mQ_{\Gr(2,4)} \otimes \mU_{\Gr(2,4)}(0,1) \oplus \mQ_{\PP^2}(1,0)
;
\end{array}
\end{equation*}
the only non-zero cohomology groups are 
$h^1(\mF^\vee \otimes T_X)=1$ and $h^0(T_X)=23$, which yield $h^0(T_X|_Y)=24$. Similarly, the term on the right is $\mF|_Y$, which is resolved by a Koszul complex whose terms are
\begin{equation*}
\begin{array}{rcl}
\W^3 \mF^\vee \otimes \mF &= & 
\of(-2,-1) \oplus \mU_{\Gr(2,4)}(-1,-2)
\\
\W^2 \mF^\vee \otimes \mF & = & 
\mU_{\Gr(2,4)}(-1,0) \oplus \of(-2,1) \oplus \Sym^2 \mU_{\Gr(2,4)}(0,-1) \oplus \of(0,-2) \oplus \mU_{\Gr(2,4)}(-1,0)
\\
\mF^\vee \otimes \mF & = & 
\of^{\oplus 2} \oplus \mU_{\Gr(2,4)}(-1,2) \oplus \mU_{\Gr(2,4)}(1,-1) \oplus \Sym^2 \mU_{\Gr(2,4)}(0,1)
\\
\mF & = & 
\of(0,2) \oplus \mU_{\Gr(2,4)}(1,1)
;
\end{array}
\end{equation*}
the only non-zero cohomology groups are 
$h^2(\W^2 \mF^\vee \otimes \mF)=1$, $h^0(\mF^\vee \otimes \mF)=2$, and $h^0(\mF)=32$, which yield $h^0(\mF|_Y)=31$.

Thus, $h^1(T_Y)-h^0(T_Y)=31-24=7$, and indeed Fano \hyperlink{Fano2--16}{2--16} is known to have a $7$-dimensional moduli space.

\section{Fano 3-folds as zero loci of sections}
\label{Fano3folds}

In this section a model for each Fano 3-fold as the zero locus of a general section of a vector bundle over a product of Grassmannians is given, provided that such a description is not available in the literature. For each model we prove the identification with the corresponding Fano; all the examples have been checked to have the right Hodge diamond and invariants as described in Section \ref{computeinvariants}.

\hypertarget{Fano1--1}{\subsection*{Fano 1--1}}
\subsubsection*{Mori-Mukai} Double cover of $\PP^3$ with branch locus a divisor of degree 6.
\subsubsection*{Our description} $\mZ(\of(6)) \subset \PP(1,1,1,1,3)$.
\subsubsection*{Identification} The obvious description as weighted hypersurface is classical. We want to give however another description embedded in a product of non-weighted Grassmannians.

By Lemma \ref{lem:doublecovers}, we can express our Fano as the zero locus of $\of(2)$ over $\PP_{\PP^3}(\of(-3) \oplus \of)$. We thus need to express such projective bundle as the zero locus of a section of a suitable vector bundle. To do that, we adopt a general strategy which will be explained in more details for \hyperlink{Fano1--12}{1--12} or \hyperlink{Fano2--2}{2--2}: we start from the short exact sequences provided by Remark \ref{rem:principalParts}
\begin{equation}
\begin{gathered}
0 \rightarrow
\of(-3) \rightarrow
\of(-2)^{\oplus 4} \rightarrow
\mQ(-2) \rightarrow
0,\\
0 \rightarrow
\of(-2)^{\oplus 4} \rightarrow
\of(-1)^{\oplus 10} \rightarrow
\Sym^2\mQ(-1) \rightarrow
0,\\
\label{thirdseq}
0 \rightarrow
\of(-1)^{\oplus 10} \rightarrow
\of^{\oplus 20} \rightarrow
\Sym^3\mQ \rightarrow
0.
\end{gathered}
\end{equation}
We can arrange the first two using the snake lemma as in Lemma \ref{projBundle1-12} or Lemma \ref{projBundle2-2}: we get
\[
0 \rightarrow
\of(-3) \rightarrow
\of(-1)^{\oplus 10} \rightarrow
\Lambda \rightarrow
0
\]
for a uniquely defined extension $\Lambda \in \Ext^1(\Sym^2\mQ(-1),\mQ(-2))$. The latter sequence can be again arranged with the third one in \eqref{thirdseq}, to get
\[
0 \rightarrow
\of(-3) \rightarrow
\of^{\oplus 20} \rightarrow
K \rightarrow
0
\]
for another uniquely defined extension $K \in \Ext^1(\Sym^3\mQ,\Lambda)$. Adding $\of \rightarrow \of$ to the above sequence, we get that our Fano can be expressed as
\[
\mZ(K(0,1) \oplus \of(0,2)) \subset \PP^{3} \times \PP^{20}.
\]

\hypertarget{Fano1--12}{\subsection*{Fano 1--12}}
\subsubsection*{Mori-Mukai} Double cover of $\PP^3$ with branch locus a smooth quartic surface.
\subsubsection*{Our description} $\mZ (\of(4)) \subset \PP(1,1,1,1,2)$.
\subsubsection*{Identification}
The obvious description as weighted hypersurface is classical. We want to give however a rather simple description as a subvariety in a product of projective spaces. We notice that our Fano is, by Lemma \ref{lem:doublecovers}, the zero locus of $\of(2)$ on $\PP_{\PP^3}(\of(-2) \oplus \of)$.

\begin{lemma}
\label{projBundle1-12}
The projective bundle $\PP_{\PP^3}(\of(-2) \oplus \of)$ can be obtained as the zero locus of $\Lambda(0,1)$ over $\PP^3 \times \PP^{10}$, being $\Lambda \in \Ext^1(\Sym^2 \mQ,\mQ(-1))$ a uniquely defined extension on $\PP^3$ fitting into sequence \eqref{Lambda1--12} below.
\begin{proof}
Our goal is to write $\of_{\PP^3}(-2) \oplus \of_{\PP^3}$ as a subbundle of $\of_{\PP^3}^{\oplus 11}$. By Remark \ref{rem:principalParts}, we have two (dual) canonical short exact sequences on $\PP^3$
\begin{gather*}
0 \rightarrow \of(-2) \rightarrow \of(-1)^{\oplus 4} \rightarrow \mQ(-1) \rightarrow 0,\\
0 \rightarrow \of(-1)^{\oplus 4} \rightarrow \of^{\oplus 10} \rightarrow \Sym^2 \mQ \rightarrow 0.
\end{gather*}
These fit as the first row and middle column of the exact diagram on $\PP^3$ here below, which can be completed by the snake lemma as
\begin{equation}
\label{snake1-12}
\begin{gathered}
\xymatrix{
& 0 \ar[d] & 0 \ar[d] & 0 \ar[d]\\
0 \ar[r] & \of(-2) \ar[d]^-= \ar[r] & \of(-1)^{\oplus 4} \ar[d] \ar[r] & \mQ(-1) \ar[d] \ar[r] & 0 \\
0 \ar[r] & \of(-2)  \ar[d] \ar[r] & \of^{\oplus 10}  \ar[d] \ar[r] & \Lambda \ar[d]  \ar[r] & 0 \\
0 \ar[r] & 0 \ar[r]\ar[d] & \Sym^2 \mQ \ar[r]^-= \ar[d] & \Sym^2 \mQ \ar[r]\ar[d] & 0 \\
& 0 & 0 & 0 \\
}
\end{gathered}
\end{equation}
for a uniquely determined homogeneous vector bundle $\Lambda$ of rank $9$. The last column describes $\Lambda$ as a non-split extension
\begin{equation}
\label{Lambda1--12}
0 \rightarrow \mQ_{\PP^3}(-1) \rightarrow \Lambda \rightarrow \Sym^2 \mQ_{\PP^3} \rightarrow 0.
\end{equation}

The rank $9$ bundle $\Lambda$ is homogeneous, not completely reducible and globally generated, and its space of global sections coincides with $H^0(\PP^3, \Sym^2 \mQ)\cong \Sym^2 V_4$. Adding the middle row of \eqref{snake1-12} to $\of \rightarrow \of$, we get
\[
0 \rightarrow \of(-2)\oplus \of \rightarrow \of^{\oplus 11} \rightarrow \Lambda \rightarrow 0,
\]
whence the conclusion of the lemma.
\end{proof}
\end{lemma}

The previous lemma yields that a model for \hyperlink{Fano1--12}{1--12} is $\mZ(\Lambda(0,1) \oplus \of(0,2)) \subset \PP^3 \times \PP^{10}$.

\hypertarget{Fano2--2}{\subsection*{Fano 2--2}}
\subsubsection*{Mori-Mukai} Double cover of $\PP^1 \times \PP^2$ with branch locus a $(2,4)$ divisor.
\subsubsection*{Our description} $\mZ (\of(0,0,2) \oplus K(0,0,1)) \subset \PP^1 \times \PP^2 \times \PP^{12}$, where $K \in \Ext^2(\of(1,0)^{\oplus 6},\mQ_{\PP^2}(-1,-1))$ fits into sequences \eqref{K2-2}.
\subsubsection*{Identification}
By Lemma \ref{lem:doublecovers}, our Fano variety is the zero locus of $\of(2)$ over $\PP(\of(-1,-2) \oplus \of)$, the latter being a projective bundle on $\PP^1 \times \PP^2$. We need to express such projective bundle as the zero locus of a suitable vector bundle. 

\begin{lemma}
\label{projBundle2-2}
The projective bundle $\PP(\of(-1,-2) \oplus \of)$ can be obtained as the zero locus of $K(0,0,1)$ over $\PP^1 \times \PP^2 \times \PP^{12}$, being $K \in \Ext^2(\of(1,0)^{\oplus 6},\mQ_{\PP^2}(-1,-1))$ a uniquely defined extension on $\PP^1 \times \PP^2$ fitting into sequences \eqref{K2-2} below.
\begin{proof}
Our goal is to write $\of_{\PP^1 \times \PP^2}(-1,-2) \oplus \of_{\PP^1 \times \PP^2}$ as a subbundle of $\of_{\PP^1 \times \PP^2}^{\oplus 13}$. By Remark \ref{rem:principalParts}, we have two (dual) canonical short exact sequences on $\PP^2$
\begin{gather*}
0 \rightarrow \of(-2) \rightarrow \of(-1)^{\oplus 3} \rightarrow \mQ(-1) \rightarrow 0,\\
0 \rightarrow \of(-1)^{\oplus 3} \rightarrow \of^{\oplus 6} \rightarrow \Sym^2 \mQ \rightarrow 0.
\end{gather*}
These fit as the first row and middle column of the exact diagram on $\PP^2$ here below, which can be completed by the snake lemma as
\begin{equation}
\label{snake2-2}
\begin{gathered}
\xymatrix{
& 0 \ar[d] & 0 \ar[d] & 0 \ar[d]\\
0 \ar[r] & \of(-2) \ar[d]^-= \ar[r] & \of(-1)^{\oplus 3} \ar[d] \ar[r] & \mQ(-1) \ar[d] \ar[r] & 0 \\
0 \ar[r] & \of(-2)  \ar[d] \ar[r] & \of^{\oplus 6}  \ar[d] \ar[r] & \Lambda \ar[d]  \ar[r] & 0 \\
0 \ar[r] & 0 \ar[r]\ar[d] & \Sym^2 \mQ \ar[r]^-= \ar[d] & \Sym^2 \mQ \ar[r]\ar[d] & 0\\
& 0 & 0 & 0 
}
\end{gathered}
\end{equation}
for a uniquely determined homogeneous vector bundle $\Lambda$ of rank $5$. The last column describes $\Lambda$ as a non-split extension
\begin{equation}
\label{LambdaProvv}
0 \rightarrow \mQ_{\PP^2}(-1) \rightarrow \Lambda \rightarrow \Sym^2 \mQ_{\PP^2} \rightarrow 0.
\end{equation}

We can pull back the middle row of \eqref{snake2-2} on $\PP^1 \times \PP^2$ and twist it by $\of(-1,0)$. This and the standard (pulled back) Euler sequence on $\PP^1$ can be inserted as the first row and second column in the exact diagram below, which can be again completed by the snake lemma as
\begin{equation}
\label{KSnake}
\begin{gathered}
\xymatrix{
& 0 \ar[d] & 0 \ar[d] & 0 \ar[d]\\
0 \ar[r] & \of(-1,-2) \ar[d]^-= \ar[r] & \of(-1,0)^{\oplus 6} \ar[d] \ar[r] & \Lambda(-1,0) \ar[d] \ar[r] & 0 \\
0 \ar[r] & \of(-1,-2)  \ar[d] \ar[r] & \of^{\oplus 12}  \ar[d] \ar[r] & K \ar[d]  \ar[r] & 0 \\
0 \ar[r] & 0 \ar[r]\ar[d] & \of(1,0)^{\oplus 6} \ar[r]^-= \ar[d] & \of(1,0)^{\oplus 6} \ar[r]\ar[d] & 0 \\
& 0 & 0 & 0 \\
}
\end{gathered}
\end{equation}
for a uniquely determined homogeneous vector bundle $K$ of rank $11$. We can further describe $K$ as an element of $\Ext^2(\of(1,0)^{\oplus 6},\mQ_{\PP^2}(-1,-1))$ obtained by combining the short exact sequences
\begin{equation}
\begin{gathered}
\label{K2-2}
0 \rightarrow \mQ_{\PP^2}(-1,-1) \rightarrow \Lambda(-1,0) \rightarrow \Sym^2 \mQ_{\PP^2}(-1,0) \rightarrow 0,
\\
0 \rightarrow \Lambda(-1,0) \rightarrow K \rightarrow \of(1,0)^{\oplus 6} \rightarrow 0.
\end{gathered}
\end{equation}

The bundle $K$ is homogeneous, not completely reducible and globally generated, and its space of global sections coincides with $H^0(\PP^1 \times \PP^2, \of(1,0)^{\oplus 6})$. Adding the middle row of \eqref{KSnake} to $\of \rightarrow \of$, we get
\[
0 \rightarrow \of(-1,-2)\oplus \of \rightarrow \of^{\oplus 13} \rightarrow K \rightarrow 0,
\]
whence the conclusion of the lemma.
\end{proof}
\end{lemma}

By construction, the bundle $\of(2)$ on $\PP(\of(-1,-2) \oplus \of)$ is identified with $\of(0,0,2)$ over the zero locus of $K$ on $\PP^1 \times \PP^2 \times \PP^{12}$, whence the conclusion.

\hypertarget{Fano2--3}{\subsection*{Fano 2--3}}
\subsubsection*{Mori-Mukai} Blow up of \hyperlink{Fano1--12}{1--12} in an elliptic curve which is the intersection of two divisors from $|-\frac{1}{2}K|$.
\subsubsection*{Our description} $\mZ (\of(4,0) \oplus \of(1,1)) \subset \PP(1,1,1,1,2) \times \PP^1$.

\subsubsection*{Identification} The first bundle on $\PP(1,1,1,1,2)$ gives \hyperlink{Fano1--12}{1--12}. We can conclude by Lemma \ref{lem:blowup}.

It is possible to provide a rather simple description involving only projective spaces. To do this, recall that a model for \hyperlink{Fano1--12}{1--12} is $\mZ(\Lambda(0,1) \oplus \of(0,2)) \subset \PP^3 \times \PP^{10}$. The adjunction formula shows that the canonical divisor is the restriction of $\of(-2,0)$; by Lemma \ref{lem:blowup}, a model for \hyperlink{Fano2--3}{2--3} is therefore given by
\[
\mZ(\Lambda(0,1,0) \oplus \of(0,2,0) \oplus \of(1,0,1)) \subset \PP^3 \times \PP^{10} \times \PP^1.
\]

\hypertarget{Fano2--5}{\subsection*{Fano 2--5}}
\subsubsection*{Mori-Mukai} Blow up of 1--13 in a plane cubic.
\subsubsection*{Our description} $\mZ(\of(0,3) \oplus \of(1,1)) \subset \PP^1 \times \PP^4$.
\subsubsection*{Identification} The first bundle on $\PP^4$ gives 1--13. We conclude by Lemma \ref{lem:blowup}.

\hypertarget{Fano2--8}{\subsection*{Fano 2--8}}
\subsubsection*{Mori-Mukai} Double cover of \hyperlink{Fano2--35}{2--35} with branch locus an anticanonical divisor such that the intersection with the exceptional divisor is smooth.

\subsubsection*{Our description} $\mZ(\Lambda(0,0,1) \oplus \of(0,0,2)) \subset \PP^2 \times \PP^3 \times \PP^{12}$, being $\Lambda \in \Ext^1(\mQ_{\PP^3}^{\oplus 3},\mQ_{\PP^2}(0,-1))$ a uniquely defined extension on $\PP^2 \times \PP^3$ fitting into sequence \eqref{Lambda2-8} below.
\subsubsection*{Identification}
As shown below, $Y:={}$ \hyperlink{Fano2--35}{2--35} can be obtained as $\mZ(\mQ_{\PP^2}(0,1))  \subset \PP^2 \times \PP^3$. By Lemma \ref{lem:doublecovers}, our Fano variety is the zero locus of $\of(2)$ on $\PP_Y(\of(-1,-1) \oplus \of)$.

As it turns out, the projective bundle $\PP_{\PP^2 \times \PP^3}(\of(-1,-1) \oplus \of)$ can be obtained as the zero locus of $\Lambda(0,0,1)$ over $\PP^2 \times \PP^3 \times \PP^{12}$, being $\Lambda \in \Ext^1(\mQ_{\PP^3}^{\oplus 3},\mQ_{\PP^2}(0,-1))$ a uniquely defined extension on $\PP^2 \times \PP^3$ fitting into sequence \eqref{Lambda2-8} below. To see it, we can argue as in Lemma \ref{projBundle1-12} or Lemma \ref{projBundle2-2}: we combine the (pull back of the) two (possibly twisted) Euler sequences
\begin{gather*}
0 \rightarrow \of(-1,-1) \rightarrow \of(0,-1)^{\oplus 3} \rightarrow \mQ_{\PP^2}(0,-1) \rightarrow 0,\\
0 \rightarrow \of(0,-1)^{\oplus 3} \rightarrow \of^{\oplus 12} \rightarrow \mQ_{\PP^3}^{\oplus 3} \rightarrow 0.
\end{gather*}
We get
\begin{gather}
\label{inclusion2-8}
0 \rightarrow \of(-1,-1) \rightarrow \of^{\oplus 12} \rightarrow \Lambda \rightarrow 0,
\\
\label{Lambda2-8}
0 \rightarrow \mQ_{\PP^2}(0,-1) \rightarrow \Lambda \rightarrow \mQ_{\PP^3}^{\oplus 3} \rightarrow 0,
\end{gather}
where the rank $11$ bundle $\Lambda$ is  homogeneous, not completely reducible and globally generated, and its space of global sections coincides with $H^0(\PP^3, 
\mQ^{\oplus 3}) \cong (V_4)^{\oplus 3}$. Adding $\of \rightarrow \of$ to \eqref{inclusion2-8} we get the desired description for $\PP_{\PP^2 \times \PP^3}(\of(-1,-1) \oplus \of)$ and the conclusion.

\hypertarget{Fano2--10}{\subsection*{Fano 2--10}}
\subsubsection*{Mori-Mukai} Blow up of 1--14 in an elliptic curve which is an intersection of 2 hyperplanes.
\subsubsection*{Our description} $\mZ(\of(2,0) \oplus \of(1,1))\subset \Gr(2,4) \times \PP^1$.
\subsubsection*{Identification} See Lemma \ref{lem:blowup}.

\hypertarget{Fano2--11}{\subsection*{Fano 2--11}}
\subsubsection*{Mori-Mukai} Blow up of 1--13 in a line.
\subsubsection*{Our description} $\mZ(\mQ_{\PP^2}(0,1) \oplus \of(1,2)) \subset \PP^2 \times \PP^4$.
\subsubsection*{Identification} By Lemma \ref{lem:blow}, the zero locus of the first summand gives $\Bl_{\PP^1}(\PP^4)$. Let $\PP^4=\PP(V_5)$ with dual coordinates $x_0, \dotsc, x_4 \in V_5^\vee$. Assume that $\PP^1=\PP(V_2)$ is given by the vanishing of $x_2, \dotsc, x_4$. A general section in $H^0(\PP^2 \times \PP^4,\of(1,2))$ is identified with a cubic in $\Sym^3(V_5^\vee)/\Sym^3(V_2^\vee)$, i.e., a cubic without terms in $x_0^3, x_0^2x_1, x_0x_1^2,x_1^3$. Such cubic contains $\PP(V_2)$, hence the claim.
Notice that, using the equivalent Corollary \ref{cor:blowupflag}, we can describe \hyperlink{Fano2--11}{2--11} as well as the zero locus $\mZ(\mQ_2^{\oplus 2} \oplus \of(2,1)) \subset \Fl(1,3,5)$.

\hypertarget{Fano2--15}{\subsection*{Fano 2--15}}
\subsubsection*{Mori-Mukai}Blow up of $\PP^3$ in the intersection of a quadric and a cubic where the quadric is smooth.
\subsubsection*{Our description} $\mZ(\mQ_{\PP^3}(0,1) \oplus \of(2,1)) \subset \PP^3 \times \PP^4$.
\subsubsection*{Identification}

By Lemma \ref{lem:blowDegeneracyLocus}, our Fano is the zero locus of $\of(1) \otimes \pi^*\of(2)$ over $\pi:\PP(\of(-1)\oplus \of) \rightarrow \PP^3$.

Adding $\of \rightarrow \of$ to the standard Euler sequence on $\PP^3$ we get
\[
0 \rightarrow
\of(-1) \oplus \of \rightarrow
\of^{\oplus 5} \rightarrow
\mQ \rightarrow
0,
\]
whence the result.

Another simple description of our Fano is
\begin{equation}
\label{anotherDescription}
\mZ(\mQ_2 \oplus \of(1,2)) \subset \Fl(1,2,5);
\end{equation}
the two descriptions are equivalent thanks to the correspondence between subvarieties of flags and of products of Grassmannians given by Lemma \ref{lem:identificationsOnFlags}, Remark \ref{rmk:wisniewski}, and Lemma \ref{lem:blow}. From these one can immediately identify $(\mZ(\mQ_2) \subset \Fl(1,2,V_5)) \cong \PP_{\PP(V_5/v_0)}(\of(-1) \oplus (v_0 \otimes\of)) \cong \Bl_{\PP(v_0)} \PP(V_5)$ as $\mZ(\mQ_{\PP^3}\boxtimes \of_{\PP^4}(1)) \subset \PP(V_5/v_0) \times \PP((V/v_0) \oplus v_0)$. On the latter we have that $\of(1,0) \cong p^*\of_{\PP^3}(1)$ and $\of(0,1) \cong \pi^*\of_{\PP^4}(1)$, where $p$ is the projective bundle map and $\pi$ the blow up map.

We want to provide a direct way to describe our Fano as \eqref{anotherDescription}, in order to show how the Cayley trick can be effectively used. First note that by Corollary \ref{cor:onF12n} $X=\mZ(\mQ_2) \subset F:=\Fl(1,2,5)$ is identified with $\PP_{\PP^3}(\of(-1) \oplus \of) \cong \PP_{\PP^3}(\of(-3) \oplus \of(-2)) = \PP(E)$, with $E:=\of_{\PP^3}(-3) \oplus \of_{\PP^3}(-2)$. To use Corollary \ref{cor:cayleycrit} we want to show that 
\[
H^0(X, \of_X(1,2)) \cong H^0(\PP(E), \of_{\PP(E)}(1)) \cong H^0(\PP^3, \of_{\PP^3}(2) \oplus \of_{\PP^3}(3)).
\]
In order to compute $H^0(X, \of_X(1,2))$ we use the Koszul complex for $X \subset F$ twisted by $\of_{F}(1,2)$.  The only non-zero cohomology groups are
\[
\begin{array}{ll}
H^0(F, \W^3 \mQ^{\vee}_2 \otimes \of_F(1,2)) \cong \Schur_{2,2,2,1}V_5 \cong \C^{40},&
H^0(F, \W^2 \mQ^{\vee}_2 \otimes \of_F(1,2)) \cong \Schur_{3,2,2,1}V_5 \cong \C^{175}\\
\rule{0pt}{12pt}
H^0(F, \mQ^{\vee}_2 \otimes \of_F(1,2)) \cong \Schur_{3,3,2,1}V_5 \cong \C^{280},&
H^0(F,  \of_F(1,2)) \cong \Schur_{3,3,3,1}V_5 \cong \C^{175}.
\end{array}
\]
As in Lemma \ref{lem:identificationsOnFlags}, $\mU_2|_X= \overline{\mU_1} \oplus \of$: this is therefore equivalent to split $V_5 = V_4 \oplus \C v_0$, and apply the above Schur functors to a such decomposed $V_5$ to get $\SL(4)\times \C^*$ representations, with the $\C^*$ component being the trivial representation. As it turns out,
\[
\begin{array}{rcl}
\Schur_{2,2,2,1}(V_4 \oplus \C) & = & \Schur_{2,2,1}V_4\oplus\Schur_{2,2,2}V_4\oplus\Schur_{2,2,1,1}V_4\oplus\Schur_{2,2,2,1}V_4, \\
\rule{0pt}{12pt} \Schur_{3,2,2,1}(V_4 \oplus \C) & = & \Schur_{3,2,2,1}V_4\oplus\Schur_{3,2,2}V_4\oplus\Schur_{3,2,1,1}V_4\oplus\Schur_{3,2,1}V_4\oplus\Schur_{2,2,2,1}V_4\oplus\Schur_{2,2,2}V_4\oplus\\ & & {}\oplus \Schur_{2,2,1,1}V_4\oplus\Schur_{2,2,1}V_4,\\
\rule{0pt}{12pt} \Schur_{3,3,2,1}(V_4 \oplus \C)& = &\Schur_{3,3,2,1}V_4\oplus\Schur_{3,3,2}V_4\oplus\Schur_{3,3,1,1}V_4\oplus\Schur_{3,3,1}V_4\oplus\Schur_{3,2,2,1}V_4\oplus  \Schur_{3
      ,2,2}V_4\oplus\\ & &{}\oplus\Schur_{3,2,1,1}V_4\oplus\Schur_{3,2,1}V_4,\\
\rule{0pt}{12pt} \Schur_{3,3,3,1}(V_4 \oplus \C)& = &\Schur_{3,3,3,1}V_4\oplus\Schur_{3,3,3}V_4\oplus\Schur_{3,3,2,1}V_4\oplus\Schur_{3,3,2}V_4\oplus\Schur_{3,3,1,1}V_4\oplus\Schur_{3
      ,3,1}V_4.
\end{array}
\]

Therefore, splitting the Koszul complex in short exact sequences, we get the natural isomorphism (of vector spaces)
\[H^0(X, \of_X(1,2)) \cong \Schur_{2,2,2}V_4 \oplus \Schur_{3,3,3}V_4 \cong \Sym^2 V_4^{\vee} \oplus \Sym^3 V_4^{\vee} \cong H^0(\PP^3, \of_{\PP^3}(2) \oplus \of_{\PP^3}(3)),\]
as claimed. It suffices to use Corollary \ref{cor:cayleycrit} to show that $X$ coincides with the Mori--Mukai description as the blow up of $\PP^3$ in the complete intersection of a quadric and a cubic surfaces.

\hypertarget{Fano2--16}{\subsection*{Fano 2--16}}
\subsubsection*{Mori-Mukai} Blow up of 1--14 in a conic.
\subsubsection*{Our description} $\mZ(\of(1,0) \oplus \of(0,2)) \subset \Fl(1,2,4)$.
\subsubsection*{Identification} Let $Y=\mZ(\of_F(0,2)) \subset \Fl(1,2,4)$ and $\tY= \mZ(\of_G(2)) \subset \Gr(2,4)$. One directly checks that \[H^0(Y, \of_Y(1,0)) \cong H^0(\tY, \mU^{\vee}|_{\tY}).\]
In fact, both spaces can be naturally identified with $V_4^{\vee}$, as in \hyperlink{Fano2--15}{2--15}. Then it suffices to apply Corollary \ref{cor:cayleycrit} to get that $X=\mZ(\of_Y(1,0)) \subset Y \cong \Bl_{\mZ( \mU^{\vee}|_{\tY})}\tY$, where we used that by duality on $\Gr(2,4)$, $\mU(1) \cong \mU^{\vee} \cong (\pi_2)_{*} \of_F(1,0)$. We conclude the proof by noting that $\tY$ is a complete intersection of two quadrics in $\PP^5$, and $(\mZ( \mU^{\vee}|_{\tY}) \subset \tY) = \mZ(\mU^{\vee} \oplus \of_G(2)) \subset \Gr(2,4)$ which is a plane conic.

We want to give an alternative description of this Fano in the product of two Grassmannians. For this, let us start by the Mori--Mukai description. Lemma \ref{lem:blowInGrass} enables us to describe $\Bl_{\PP^2} \Gr(2,4)$ in the product $(\PP^2)^{\vee} \times \Gr(2,4)$. We then need to cut with an extra quadric intersecting the blown up $\PP^2$. As we are going to see in full details for \hyperlink{Fano2--26}{2--26}, for this it suffices to take a section of $\of(0,2)$. Summarising, we can describe our \hyperlink{Fano2--16}{2--16} as 
\[ \mZ(\mU^{\vee}_{\Gr(2,4)}(1,0) \oplus \of(0,2)) \subset \PP^2 \times \Gr(2,4).
\]
\hypertarget{Fano2--17}{\subsection*{Fano 2--17}}
\subsubsection*{Mori-Mukai}  Blow up of $\mathbb{Q}_3$ in an elliptic curve of degree 5.
\subsubsection*{Our description} $\mZ(\of(0,1) \oplus \of(1,1)) \subset \Fl(1,2,4)$.
\subsubsection*{Identification}
A model for this 3-fold in $\Gr(2,4) \times \PP^3$ can be found in \cite{corti}. As an exception to our self-imposed rule, we want to give here an alternative description in a flag variety, since we find it particularly nice. Let us show that our Fano can be written as
\[
\mZ(\of(1,1) \oplus \of(0,1)) \subset \Fl(1,2,4).
\]

 As before, we check that \[H^0(Y, \of_Y(1,1)) \cong H^0(\tY, \oU^{\vee}(1)|_{\tY}),\]
where we are using the same notation as above. These spaces are both $16$-dimensional and isomorphic as vector spaces to $\Schur_{2,1}V_4^{\vee}/V_4^{\vee}$, where we can interpret $\Schur_{2,1}V_4$ as the kernel of the natural contraction map $\contr: V_4 \otimes \W^2 V_4^{\vee} \rightarrow V_4^{\vee}$. These spaces of sections are not $\SL(V_4)$-representations: in fact $\tY$ (and similarly for the section on the flag) is not homogeneous for the whole group, but rather for $\SO(V_3)$, and one could write a more elegant expression for the spaces of section as in \hyperlink{Fano2--15}{2--15}. To conclude we apply Corollary \ref{cor:cayleycrit}: we have that $X=\mZ(\of_Y(1,1)) \cong \Bl_{\tZ} \tY$ where $\tY$ is a quadric 3-fold, and the centre of the blow up is $\tZ= \mZ(\mU^{\vee}(1) \oplus \of(1)) \subset \Gr(2,4)$, which can be easily checked to be an elliptic curve of degree 5.

\hypertarget{Fano2--18}{\subsection*{Fano 2--18}}
\subsubsection*{Mori-Mukai} Double cover of 2-34 with branch locus a divisor of degree $(2,2)$.
\subsubsection*{Our description} $\mZ(\Lambda(0,0,1) \oplus \of(0,0,2)) \subset \PP^1 \times \PP^2 \times \PP^6$, being $\Lambda \in \Ext^1(\mQ_{\PP^2}^{\oplus 2},\of(1,-1))$ a uniquely defined extension on $\PP^1 \times \PP^2$ fitting into sequence \eqref{Lambda2-18} below.

\subsubsection*{Identification}

By Lemma \ref{lem:doublecovers}, our Fano variety is the zero locus of $\of(2)$ on $\PP_{\PP^1 \times \PP^2}(\of(-1,-1) \oplus \of)$. As it turns out, the latter projective bundle can be obtained as the zero locus of $\Lambda(0,0,1)$ over $\PP^1 \times \PP^2 \times \PP^6$, being $\Lambda \in \Ext^1(\mQ_{\PP^2}^{\oplus 2},\of(1,-1))$ a uniquely defined extension on $\PP^1 \times \PP^2$ fitting into sequence \eqref{Lambda2-18} below. To see it, we can argue as in Lemma \ref{projBundle1-12} or Lemma \ref{projBundle2-2}: we combine the (pull back of the) two (possibly twisted) Euler sequences
\begin{gather*}
0 \rightarrow \of(-1,-1) \rightarrow \of(0,-1)^{\oplus 2} \rightarrow \of(1,-1) \rightarrow 0,\\
0 \rightarrow \of(0,-1)^{\oplus 2} \rightarrow \of^{\oplus 6} \rightarrow \mQ_{\PP^2}^{\oplus 2} \rightarrow 0.
\end{gather*}
We get
\begin{gather}
\label{inclusion2-18}
0 \rightarrow \of(-1,-1) \rightarrow \of^{\oplus 6} \rightarrow \Lambda \rightarrow 0,
\\
\label{Lambda2-18}
0 \rightarrow \of(1,-1) \rightarrow \Lambda \rightarrow \mQ_{\PP^2}^{\oplus 2} \rightarrow 0,
\end{gather}
where the rank $5$ bundle $\Lambda$ is homogeneous, not completely reducible and globally generated, and its space of global sections coincides with $H^0(\PP^2, 
\mQ^{\oplus 2}) \cong (V_3)^{\oplus 2}$. Adding $\of \rightarrow \of$ to \eqref{inclusion2-18} we get the desired description for $\PP_{\PP^1 \times \PP^2}(\of(-1,-1) \oplus \of)$ and the conclusion.

\hypertarget{Fano2--19}{\subsection*{Fano 2--19}}
\subsubsection*{Mori-Mukai} Blow up of 1--14 in a line.
\subsubsection*{Our description}  $\mZ(\mQ_{\PP^3}(0,1) \oplus \of(1,1)^{\oplus 2} ) \subset \PP^3 \times \PP^5$. 

\subsubsection*{Identification} It suffices to apply Lemma \ref{lem:blow} and argue as done for \hyperlink{Fano2--11}{2--11}. The zero locus of the first factor identifies $\mZ(\mQ_{\PP^3}(0,1))$ with the blow up $\Bl_{\PP^1}(\PP^5)$. Let $\PP^5=\PP(V_6)$ with dual coordinates $x_0, \dotsc, x_5 \in V_6^\vee$. Assume that $\PP^1=\PP(V_2)$ is given by the vanishing of $x_2, \dotsc, x_5$. A general section in $H^0(\PP^3 \times \PP^5,\of(1,1)^2)$ is identified with two quadrics in $\Sym^2(V_6^\vee)/\Sym^2(V_2^\vee)$, i.e., quadrics without terms in $x_0^2, x_1^2, x_0x_1$. Such quadrics have generically maximal rank, so their intersection is smooth and contains $\PP(V_2)$, hence the claim.
Notice that, using the equivalent Corollary \ref{cor:blowupflag}, we can describe \hyperlink{Fano2--19}{2--19} as the zero locus of 
$\mZ(\mQ_2^{\oplus 2} \oplus \of(1,1)^{\oplus 2}) \subset \Fl(1,3,6)$ as well.

\hypertarget{Fano2--22}{\subsection*{Fano 2--22}}
\subsubsection*{Mori-Mukai} Blow up of $\mathbb{V}_5$ in a conic.
\subsubsection*{Our description} $\mZ(\mQ_{\Gr(2,5)}(1,0) \oplus \of(0,1)^{\oplus 3}) \subset \PP^3 \times \Gr(2,5).$
\subsubsection*{Identification} In \cite{corti} this variety is described as $\mZ(\of(1,0) \oplus \of(0,1)^{\oplus 3}) \subset \Fl(1,2,5)$. This description is equivalent to the one given here simply applying Lemma \ref{lem:blowInGrass} with $k=3$ (where we identify $\Gr(3,4)$ and $\Gr(3,5)$ with $\PP^3$ and $\Gr(2,5)$). The three residual sections of $\of(0,1)$ cut both $\Gr(2,5)$ (in $\mathbb{V}_5$) and $\Gr(2,4)$ (in a conic).

\hypertarget{Fano2--23}{\subsection*{Fano 2--23}}
\subsubsection*{Mori-Mukai}  Blow up of $\mathbb{Q}_3$ in an intersection of a hyperplane and a quadric.
\subsubsection*{Our description} $\mZ(\mQ_2 \oplus \of(1,1) \oplus \of(0,2)) \subset \Fl(1,2,6)$.

\subsubsection*{Identification} We apply Corollary \ref{cor:cayleycrit}. In the notation of the corollary, we denote by $Y \subset \Fl(1,2,6)$ the zero locus of $\mQ_2 \oplus \of(0,2)$. We identify $\widetilde{Y}$ with a three dimensional quadric $Q \subset \PP^4$. What we need to check is \[H^0(Y, \of_Y(1,1)) \cong H^0(Q, \of_Q(1)) \oplus H^0(Q, \of_Q(2)). \]

To verify this, one can argue as for \hyperlink{Fano2--15}{2--15}: one can compute the $\SL(V_6)$-representations arising from the Koszul complex resolving $\of_Y(1,1)$. These representations, when seen as $\SL(V_5)\times \C^*$-representations under the splitting $V_6 = V_5 \oplus \C v_0$, sum up to $\Sigma_{1,1,1,1}V_5 \oplus \Sigma_{2,2,2,2}V_5/ \C$, which is clearly isomorphic to the right hand side.

Therefore $X \cong \Bl_{\widetilde{Z}} Q$, where $\widetilde{Z}$ is given by the intersection of a quadratic and linear forms in $Q$.

We provide the following alternative description for this Fano:
\[ \mZ(\mQ_{\PP^4}(0,1) \oplus \of(2,0) \oplus \of(1,1)) \subset \PP^4 \times \PP^5,
\]
which can be shown to be equivalent to the previous one following the same lines of \hyperlink{Fano2--15}{2--15}.

\hypertarget{Fano2--26}{\subsection*{Fano 2--26}}
\subsubsection*{Mori-Mukai} Blow up of $\mathbb{V}_5$ in a curve of genus 0.
\subsubsection*{Our description} $\mZ(\mQ_{2,4} \boxtimes \mU^{\vee}_{2,5} \oplus \of(1,0) \oplus \of(0,1)^{\oplus 2}) \subset \Gr(2,4) \times \Gr(2,5).$

\subsubsection*{Identification} 
By Lemma \ref{lem:blowInGrass} we can identify $\mZ(\mQ_{2,4} \boxtimes \mU^{\vee}_{2,5}) \subset \Gr(2,4) \times \Gr(2,5)$ as the blow up $\Bl_{\PP^3} \Gr(2,5)$, where $\PP^3$ is identified with $\mZ(\mQ) \subset \Gr(2,5) \cong \Gr(2,V_5)$, given by a vector $w \in V_5^\vee$.

Without loss of generality, we can assume $w=x_0$. We have a splitting $V_5^{\vee}= x_0 \oplus W_4$ that induces a splitting $\W^2 V_5^{\vee} = \W^2 W_4 \oplus x_0 \wedge W_4$. For simplicity, let us fix a basis $x_0,\dotsc,x_4$ of $V_5^{\vee}$ and the corresponding dual basis $e_0,\dotsc,e_4$ of $V_5$. The above $\PP^3$ is by definition described by the points in $\Gr(2,5)$ of the form $e_0 \wedge \alpha$, where $\alpha \in \langle e_1,\dotsc,e_4 \rangle$.

By construction, any $f \in \W^2 W_4=|\of(1,0)|$ does not contain any summand of the form $x_0 \wedge \beta$, so that $f(e_0 \wedge \alpha)=0$. In other words, $f \in \Ann(\PP^3)$, hence its zero locus in $\Bl_{\PP^3} \Gr(2,5)$ contains the whole $\PP^3$ and does not cut it. The two extra sections of $\of(0,1)$ cut $\PP^3$ in a codimension two linear subspace. Therefore our zero locus can be seen as the blow up of $\mZ(\of_{\Gr(2,V_5)}^{\oplus 3}(1))\subset \Gr(2,V_5)$ along $\PP^1 \cong \mZ(\of_{\PP^3}^{\oplus 2}(1))\subset \PP^3$.

Another description of this Fano is as $\mZ(\mU_1^{\vee} \oplus \of(1,0) \oplus \of(0,1)^{\oplus 2}) \subset \Fl(2,3,5)$. By Lemma \ref{lem:identificationsOnFlags} this can be easily identified with the alternative description of this Fano given in \cite{corti}.

\hypertarget{Fano2--28}{\subsection*{Fano 2--28}}
\subsubsection*{Mori-Mukai} Blow up of $\PP^3$ in a plane cubic.
\subsubsection*{Our description} 
$\mZ(\Lambda(0,1) \oplus \of(1,1)) \subset \PP^3 \times \PP^{10}$, being $\Lambda \in \Ext^1(\Sym^2 \mQ,\mQ(-1))$ a uniquely defined extension on $\PP^3$ fitting into sequence \eqref{Lambda1--12} above.

\subsubsection*{Identification}
Our Fano variety is the blow up of $\PP^3$ along the intersection of two divisors of degree $1$ and $3$. By Lemma \ref{lem:blowDegeneracyLocus}, it corresponds to the zero locus of $\pi^*\of_{\PP^3}(1) \otimes \of(1)$ over the projective bundle $\pi:\PP(\of(-2)\oplus \of)\rightarrow \PP^3$. We conclude by Lemma \ref{projBundle1-12}.

\hypertarget{Fano2--29}{\subsection*{Fano 2--29}}
\subsubsection*{Mori-Mukai} Blow up of $\mathbb{Q}_3$ in a conic.
\subsubsection*{Our description} $\mZ(\of(0,2) \oplus \of(1,1)) \subset \PP^1 \times \PP^4$.

\subsubsection*{Identification} See Lemma \ref{lem:blowup}.

\hypertarget{Fano2--30}{\subsection*{Fano 2--30}}
\subsubsection*{Mori-Mukai} Blow up of $\PP^3$ in a conic.
\subsubsection*{Our description} $\mZ(\mQ_2 \oplus \of(1,1)) \subset \Fl(1,2,5)$.

\subsubsection*{Identification} We apply Corollary \ref{cor:cayleycrit}. Following the notation of the corollary, we denote by $Y \subset \Fl(1,2,5)$ the zero locus of $\mQ_2$ and we identify $\widetilde{Y}$ with a $\PP^3$. What we need to check is \[H^0(Y, \of_Y(1,1)) \cong H^0(\PP^3, \of_{\PP^3}(1)\oplus \of_{\PP^3}(2)). \]

To verify this, one can argue as for \hyperlink{Fano2--15}{2--15} or \hyperlink{Fano2--23}{2--23}: the representations arising from the Koszul complex resolving $\of_Y(1,1)$, when seen as $\SL(V_4)\times \C^*$-representations, sum up to $\Sigma_{1,1,1}V_4 \oplus \Sigma_{2,2,2}V_4$, which is clearly isomorphic to the right hand side.

Notice that as an alternative description we can follow the same lines of \hyperlink{Fano2--15}{2--15} and describe the Fano \hyperlink{Fano2--30}{2--30} as 
\[
\mZ(\mQ_{\PP^3}(0,1) \oplus \of(1,1)) \subset \PP^3 \times \PP^4.
\]
\hypertarget{Fano2--31}{\subsection*{Fano 2--31}}
\subsubsection*{Mori-Mukai} Blow up of $\mathbb{Q}_3$ in a line.

\subsubsection*{Our description} $\mZ( \mU^{\vee}_{\Gr(2,4)}(1,0) \oplus \of(0,1)) \subset \PP^2  \times \Gr(2,4) $.
\subsubsection*{Identification}
We may regard $\PP^2$ as $\Gr(2,3)$, so that our Fano is given as $\mZ(\mQ \boxtimes \mU^{\vee} \oplus \of(0,1))$. Then we argue as for \hyperlink{Fano2--26}{2--26}. By Lemma \ref{lem:blowInGrass} we can identify $\mZ(\mQ \boxtimes \mU^{\vee}) \subset \Gr(2,3) \times \Gr(2,4)$ as $\Bl_{\PP^2} \Gr(2,4)$, where $\PP^2$ is identified with $\mZ(\mQ) \subset \Gr(2,4)$. As shown for \hyperlink{Fano2--26}{2--26}, the remaining section of $\of(0,1)$ cuts such $\PP^2$ in a codimension one linear subspace and the ambient $\Gr(2,4)$ in a three-dimensional quadric, hence the conclusion.

\hypertarget{Fano2--33}{\subsection*{Fano 2--33}}
\subsubsection*{Mori-Mukai} Blow up of $\PP^3$ in a line.
\subsubsection*{Our description} $\mZ(\of(1,1)) \subset \PP^1 \times \PP^3$.
\subsubsection*{Identification} See Lemma \ref{lem:blowup}.

\hypertarget{Fano2--35}{\subsection*{Fano 2--35}}
\subsubsection*{Mori-Mukai} $\Bl_p \PP^3$ or $\PP_{\PP^2}(\of \oplus \of(-1))$.
\subsubsection*{Our description}  $\mZ(\mQ_{\PP^2}(0,1)) \subset \PP^2 \times \PP^3$.
\subsubsection*{Identification}
This is a straightforward application of Lemma \ref{lem:blow}. Notice that equivalently we could describe \hyperlink{Fano2--35}{2--35} as $\mZ(\mQ_2) \subset \Fl(1,2,4)$.

\hypertarget{Fano2--36}{\subsection*{Fano 2--36}}
\subsubsection*{Mori-Mukai} $\PP_{\PP^2}(\of \oplus \of(-2))$.
\subsubsection*{Our description} $\mZ(\Lambda(0,1)) \subset \PP^2 \times \PP^6$, being $\Lambda \in \Ext^1(\Sym^2 \mQ,\mQ(-1))$ a uniquely defined extension on $\PP^2$ fitting into sequence \eqref{Lambda2-36}.

\subsubsection*{Identification}
We argue as in Lemma \ref{projBundle1-12}, with the appropriate changes. By Remark \ref{rem:principalParts}, we have two (dual) canonical short exact sequences on $\PP^2$
\begin{gather*}
0 \rightarrow \of(-2) \rightarrow \of(-1)^{\oplus 3} \rightarrow \mQ(-1) \rightarrow 0,\\
0 \rightarrow \of(-1)^{\oplus 3} \rightarrow  \of^{\oplus 6} \rightarrow \Sym^2 \mQ \rightarrow 0.
\end{gather*}
We combine them and get
\begin{gather}
\label{inclusion2-36}
0 \rightarrow \of(-2) \rightarrow \of^{\oplus 6} \rightarrow \Lambda \rightarrow 0,
\\
\label{Lambda2-36}
0 \rightarrow \mQ(-1) \rightarrow \Lambda \rightarrow \Sym^2 \mQ\rightarrow 0,
\end{gather}
where the rank $5$ bundle $\Lambda$ is homogeneous, not completely reducible and globally generated, and its space of global sections coincides with $H^0(\PP^2, \Sym^2 \mQ_{\PP^2}) \cong \Sym^2 V_3$. Adding $\of \rightarrow \of$ to \eqref{inclusion2-36} we get the desired description for $\PP(\of(-2) \oplus \of)$.

\hypertarget{Fano3--1}{\subsection*{Fano 3--1}}
\subsubsection*{Mori-Mukai} Double cover of $\PP^1 \times \PP^1 \times \PP^1$ with branch locus a divisor of degree $(2,2,2)$.
\subsubsection*{Our description} $\mZ(K(0,0,0,1) \oplus \of(0,0,0,2)) \subset \PP^1 \times \PP^1 \times \PP^1 \times \PP^8$, where the bundle $K$ is a uniquely defined extension in $\Ext^2(\of(0,0,1)^{\oplus 4},\of(1,-1,-1))$ on $\PP^1 \times \PP^1 \times \PP^1$ fitting into the chain of extensions \eqref{K3-1} below.

\subsubsection*{Identification}
By Lemma \ref{lem:doublecovers}, our Fano variety is the zero locus of $\of(2)$ on the projective bundle $\PP_{\PP^1\times\PP^1\times \PP^1}(\of(-1,-1,-1) \oplus \of)$. As it turns out, the latter projective bundle can be obtained as the zero locus of $K(0,0,0,1)$ over $\PP^1 \times \PP^1 \times \PP^1 \times \PP^8$, being $K \in \Ext^2(\of(0,0,1)^{\oplus 4},\of(1,-1,-1))$ a uniquely defined extension on $\PP^1 \times \PP^1 \times \PP^1$ fitting into \eqref{K3-1} below. 

To see it, we can argue as in Lemma \ref{projBundle1-12} or Lemma \ref{projBundle2-2}: we combine the (pull back of the) three (possibly twisted) Euler sequences
\begin{gather*}
0 \rightarrow \of(-1,-1,-1) \rightarrow \of(0,-1,-1)^{\oplus 2} \rightarrow \of(1,-1,-1) \rightarrow 0,\\
0 \rightarrow \of(0,-1,-1)^{\oplus 2} \rightarrow \of(0,0,-1)^{\oplus 4} \rightarrow \of(0,1,-1)^{\oplus 2} \rightarrow 0,\\
0 \rightarrow \of(0,0,-1)^{\oplus 4} \rightarrow \of^{\oplus 8} \rightarrow \of(0,0,1)^{\oplus 4} \rightarrow 0.
\end{gather*}
We get
\begin{equation}
\label{Inclusion3-1}
0 \rightarrow \of(-1,-1,-1) \rightarrow \of^{\oplus 8} \rightarrow K \rightarrow 0,
\end{equation}
where the rank $7$ bundle $K$, fitting into the chain of extension \eqref{K3-1}, is homogeneous, not completely reducible and globally generated, and its space of global sections coincides with $H^0(\PP^1\times\PP^1 \times \PP^1, 
\of(0,0,1)^{\oplus 4}) \cong (V_2)^{\oplus 4}$. 
\begin{equation}
\begin{gathered}
\label{K3-1}
0 \rightarrow \of(1,-1,-1) 
\rightarrow \Lambda \rightarrow \of(0,1,-1)^{\oplus 2} \rightarrow 0,
\\
0 \rightarrow \Lambda \rightarrow K \rightarrow \of(0,0,1)^{\oplus 4} \rightarrow 0.
\end{gathered}
\end{equation}

Adding $\of \rightarrow \of$ to \eqref{Inclusion3-1} and from the previous considerations we get the conclusion.

\hypertarget{Fano3--2}{\subsection*{Fano 3--2}}
\subsubsection*{Mori-Mukai} A divisor from $|\of(2) \otimes \pi^*\of(0,1)|$ on the projective bundle $\pi:\PP(\of(-1,-1) \oplus \of^{\oplus 2} ) \rightarrow \PP^1 \times \PP^1$ such that $X \cap Y$ is irreducible, where $X$ is the Fano itself and $Y \in |\of(1)|$.

\subsubsection*{Our description} $\mZ(\Lambda(0,0,1) \oplus \of(0,1,2)) \subset \PP^1 \times \PP^1 \times \PP^5$, being $\Lambda \in \Ext^1(\of(0,1)^{\oplus 2}, \of(1, -1))$ a uniquely defined extension on $\PP^1 \times \PP^1$ fitting into \eqref{Lambda3-2} below.

\subsubsection*{Identification} 
We need to find $\PP(\of(-1,-1) \oplus \of^{\oplus 2} )$ over $\PP^1 \times \PP^1$. To do that, we argue as in Lemma \ref{projBundle1-12} or Lemma \ref{projBundle2-2}: we combine the two Euler exact sequences
\begin{gather*}
0 \rightarrow \of(-1,-1) \rightarrow \of(0,-1)^{\oplus 2} \rightarrow \of(1,-1) \rightarrow 0,\\
0 \rightarrow \of(0,-1)^{\oplus 2} \rightarrow  \of^{\oplus 4} \rightarrow \of(0,1)^{\oplus 2} \rightarrow 0,
\end{gather*}
and get

\begin{gather}
\label{inclusion3-2}
0 \rightarrow \of(-1,-1) \rightarrow \of^{\oplus 4} \rightarrow \Lambda \rightarrow 0,
\\
\label{Lambda3-2}
0 \rightarrow \of(1,-1) \rightarrow \Lambda \rightarrow \of(0,1)^{\oplus 2} \rightarrow 0.
\end{gather}
where the rank $3$ bundle $\Lambda$ is homogeneous, not completely reducible and globally generated, and its space of global sections coincides with $H^0(\PP^1 \times \PP^1, \of(0,1)^{\oplus 2}) \cong V_2^{\oplus 2}$. Adding $\of^{\oplus 2} \rightarrow \of^{\oplus 2}$ to \eqref{inclusion3-2} we get the desired description for $\PP(\of(-1,-1) \oplus \of^{\oplus 2})$ and the conclusion.

\hypertarget{Fano3--4}{\subsection*{Fano 3--4}}
\subsubsection*{Mori-Mukai} Blow up of \hyperlink{Fano2--18}{2--18} in a smooth fiber of the composition of the double cover projection to $\PP^1 \times \PP^2$ with the projection to $\PP^2$.
\subsubsection*{Our description}$\mZ(\Lambda(0,0,1,0) \oplus \of(0,0,2,0) \oplus \of(0,1,0,1)) \subset \PP^1 \times \PP^2 \times \PP^6 \times \PP^1$, where the bundle $\Lambda \in \Ext^1(\mQ_{\PP^2}^{\oplus 2},\of(1,-1))$ is a uniquely defined extension on $\PP^1 \times \PP^2$ fitting into sequence \eqref{Lambda2-18}.
\subsubsection*{Identification} The first two bundles define $Y \times \PP^1$, being $Y \subset \PP^1 \times \PP^2 \times \PP^6$ the Fano \hyperlink{Fano2--18}{2--18}. The curve on $Y$ we need to blow up is a complete intersection of two $(0,1,0)$ divisors, which cut in $Y$ the preimage of a $\PP^1$-fiber of the projection $\PP^1 \times \PP^2 \rightarrow \PP^2$. We therefore conclude by Lemma \ref{lem:blowup}.

\hypertarget{Fano3--5}{\subsection*{Fano 3--5}}
\subsubsection*{Mori-Mukai} Blow up of $\PP^1 \times \PP^2$ in a curve $C$ of degree $(5,2)$ such that $C \hookrightarrow \PP^1 \times \PP^2 \rightarrow \PP^2$ is an embedding.
\subsubsection*{Our description} $\mZ(\Lambda(0,0,1) \oplus \of(0,1,1)^{\oplus 2}) \subset \PP^1 \times \PP^2 \times \PP^7$, with $\Lambda \in \Ext^1_{\PP^1 \times \PP^2}(\mQ_{\PP^2}^{\oplus 2}, \of(1,-1))$ fitting into \eqref{Lambda2-18}.

\subsubsection*{Identification}
By Lemma \ref{lem:blowDegeneracyLocus} and Lemma \ref{lem:expectedRes} below, our Fano is the zero locus of $\pi^* (\of(0,1)^{\oplus 2}) \otimes \of(1)$ over the projective bundle $\pi:\PP(\of(-1,-1) \oplus \of^{\oplus 2}) \rightarrow \PP^1 \times \PP^2$. We thus need to find the latter projective bundle as the zero locus of a suitable vector bundle.

A straightforward modification of the argument used for \hyperlink{Fano2--18}{2--18} provides the desired description: adding $\of^{\oplus 2} \rightarrow \of^{\oplus 2}$ to \eqref{inclusion2-18} we get
\[
0 \rightarrow \of(-1,-1) \oplus \of^{\oplus 2} \rightarrow \of^{\oplus 8} \rightarrow \Lambda \rightarrow 0,
\]
where $\Lambda$ fits into \eqref{Lambda2-18}. The conclusion follows as soon as we have proved the following lemma.

\begin{lemma}\label{lem:expectedRes}
The ideal sheaf of a general rational curve $C$ of bidegree $(5,2)$ in $\PP^1 \times \PP^2$ admits a locally free resolution of the form
\begin{equation}
\label{shaperesolution}
0 \rightarrow
\of(-1,-4)^{\oplus 2} \rightarrow
\of(0,-2) \oplus \of(-1,-3)^{\oplus 2} \rightarrow
\mathcal{I}_C \rightarrow 0
\end{equation}
and, conversely, a general $3 \times 2$ matrix as above yields a presentation for the ideal sheaf of a general curve $C$.
\begin{proof}
The aim is to show, on the one hand, that the above resolution is the simplest (in terms of Betti numbers) such a curve is expected to have. On the other hand, if we manage to show that a curve having that resolution exists, a semicontinuity argument yields that a general curve shares the same behaviour.

The first task requires a bit of commutative algebra, which we specialise to our setting $\PP:=\PP^1 \times \PP^2$. Let $R:=\oplus_{(a,b) \in \mathbb{Z}^2}H^0(\PP,\of(a,b))$ be the Cox ring of $\PP$. If $I_C$ denotes the ideal of $C$, which can be seens as a finitely generated $R$-module, we have a multigraded minimal free resolution
\[
0 \rightarrow F_r \rightarrow \dotso \rightarrow F_0 \rightarrow I_C \rightarrow 0,
\]
where the $F_i$ are finitely generated free modules $F_i = \oplus_{(a,b) \in \mathbb{Z}^2}R(-a,-b)^{\oplus \beta_{i,(a,b)}}$, being $\beta_{i,(a,b)}$ the so-called multigraded Betti numbers, which are independent of the chosen resolution.

The so-called multigraded Hilbert series of $I_C$ is the formal Laurent series
\[
H_{I_C}:=\sum_{(a,b) \in \mathbb{Z}^2} \dim_{\mathbb{C}}(I_C)_{(a,b)}\cdot s^a t^b,
\]
which is well-known to encode the Betti numbers $\beta_{i,(a,b)}$ in the following way: it factors as
\[
H_{I_C}=\frac{
\sum_{(a,b) \in \mathbb{Z}^2}\left( \sum_{i=0}^r (-1)^{i} \beta_{i,(a,b)}\right) \cdot s^a t^b
}{(1-s)^2(1-t)^3}.
\]

By Riemann--Roch we can compute $H^0(C,\of_{C}(a,b))$ for any $(a,b) \in \mathbb{Z}^2$; if we assume that $C$ has maximal rank, i.e., that $H^0(\PP,\of_{\PP}(a,b)) \rightarrow H^0(C,\of_{C}(a,b))$ has maximal rank for all $(a,b) \in \mathbb{Z}^2$, then we explicitly have $\dim_{\mathbb{C}}(I_C)_{(a,b)}$ and $H_{I_C}$. Straightforward computations then show that the numerator of $H_{S/I_C}$ is $t^2 + 2st^3 -2st^4$; thus, the expected resolution of $I_C$ has the shape \eqref{shaperesolution}.

To conclude, it suffices to show the existence of a curve with the right genus and degree having the desired resolution. This task can be rather difficult, depending on the given invariants: on $\PP^1 \times \PP^2$, different approaches can be adopted, such as liaison theory or the construction of the Hartshorne--Rao module of the curve, see, e.g., \cite{KeneshlouTanturri1,KeneshlouTanturri2}. Our situation, however, is favourable, as the minors of a general matrix
\[
\of(-1,-4)^{\oplus 2} \rightarrow
\of(0,-2) \oplus \of(-1,-3)^{\oplus 2}
\]
generate the ideal of a smooth curve of maximal rank with the desired invariants. This can be checked via any computer algebra software like \cite{M2}.
\end{proof}
\end{lemma}

Let $\mathcal{F}:=\Lambda(0,0,1) \oplus \of(0,1,1)^{\oplus 2}$. If we consider the normal sequence for $Y=\mZ(\mathcal{F})$ inside $\PP:=\PP^1 \times \PP^2 \times \PP^7$, a few cohomology computations via the Koszul complex as described in Section \ref{computeinvariants} provide that $h^0(T_{\PP}|_Y)=74, h^0(\mathcal{F}|_Y)=79$ and the higher cohomology groups vanish. In \cite[Corollary 8.8]{pcs} it is shown that the family of Fano \hyperlink{Fano3--5}{3--5} has a unique member with infinite automorphism group. This means that a general model $Y$ admits a $(79-74=5)$-dimensional family of deformations, which is the dimension of the moduli of Fano \hyperlink{Fano3--5}{3--5}, hence $Y$ is general in moduli.

\hypertarget{Fano3--6}{\subsection*{Fano 3--6}}
\subsubsection*{Mori-Mukai} Blow up of $\PP^3$ in the disjoint union of a line and an elliptic curve of degree 4.
\subsubsection*{Our description} $\mZ(\of(1,0,2) \oplus \of(0,1,1)) \subset \PP^1 \times \PP^1 \times \PP^3$.
\subsubsection*{Identification}
A quartic elliptic curve is a given by a complete intersections of two quadrics in $\PP^3$. It then suffices to apply twice Lemma \ref{lem:blowup}.

\hypertarget{Fano3--8}{\subsection*{Fano 3--8}}
\subsubsection*{Mori-Mukai} Divisor from the linear system $|(\alpha \circ \pi)^* (\of_{\PP^2}(1)) \boxtimes \of_{\PP^2}(2)|$ on $\Bl_p \PP^2 \times \PP^2$, where $\pi: \Bl_p \PP^2 \times \PP^2 \rightarrow \Bl_p \PP^2$ is the first projection and $\alpha: \Bl_p \PP^2 \rightarrow \PP^2$ is the blow up map. 
\subsubsection*{Our description} $\mZ(\of(0,1,2) \oplus \of(1,1,0)) \subset \PP^1 \times \PP^2 \times \PP^2$.

\subsubsection*{Identification} See Lemma \ref{lem:blowup}.

\hypertarget{Fano3--9}{\subsection*{Fano 3--9}}
\subsubsection*{Mori-Mukai} Blow up of $\PP_{\PP^2}(\of \oplus \of(-2))$ in a quartic curve on $\PP^2$.
\subsubsection*{Our description} $\mZ(\Lambda(0,1,0) \oplus \mQ_{\PP^6}(0,0,1) \oplus K(0,0,1)) \subset \PP^2 \times \PP^6 \times \PP^{20}$, where the bundle $\Lambda\in \Ext^1(\Sym^2 \mQ,\mQ(-1))$ is a uniquely defined extension on $\PP^2$ fitting into sequence \eqref{Lambda2-36} and $K \in \Ext^3(\Sym^4 \mQ,\mQ(-3))$ is a uniquely defined extension on $\PP^2$ fitting into the chain of extensions \eqref{Kappa3--9}.
\subsubsection*{Identification}
We need to blow up \hyperlink{Fano2--36}{2--36} in a quartic curve $C$ on the base $\PP^2$. The first bundle defines $Y:=$ \hyperlink{Fano2--36}{2--36} inside $\PP^2\times \PP^6$; since $C$ is the zero locus of a map
\[
\of(0,-1) \oplus \of(-4,0) \rightarrow \of
\]
on $Y$, by Lemma \ref{lem:blowDegeneracyLocus} our Fano will be the zero locus of $\of(1)$ over $\PP_Y(\of(0,-1) \oplus \of(-4,0))$.

For the first bundle $\of(0,-1)$, we have the standard (pulled back) Euler sequence
\begin{equation}
\label{of0-1}
0 \rightarrow
\of(0,-1) \rightarrow
\of^{\oplus 7} \rightarrow
\mQ_{\PP^6} \rightarrow
0;
\end{equation}
the second bundle $\of(-4,0)$ requires a cumbersome though straightforward merging of the following (dualised) short exact sequences on $\PP^2$ given by Remark \ref{rem:principalParts}:
\begin{gather*}
0 \rightarrow
\of(-4) \rightarrow
\of(-3)^{\oplus 3} \rightarrow
\mQ(-3) \rightarrow
0\\
0 \rightarrow
\of(-3)^{\oplus 3} \rightarrow
\of(-2)^{\oplus 6} \rightarrow
\Sym^2\mQ(-2) \rightarrow
0\\
0 \rightarrow
\of(-2)^{\oplus 6} \rightarrow
\of(-1)^{\oplus 10} \rightarrow
\Sym^3\mQ(-1) \rightarrow
0\\
0 \rightarrow
\of(-1)^{\oplus 10} \rightarrow
\of^{\oplus 15} \rightarrow
\Sym^4\mQ \rightarrow
0.
\end{gather*}
Arranging them repeatedly as in Lemma \ref{projBundle1-12} or Lemma \ref{projBundle2-2}, we get to a uniquely defined homogeneous rank $14$ vector bundle $K$ on $\PP^2$ which fits into
\begin{equation}
\label{of-40}
0 \rightarrow
\of(-4) \rightarrow
\of^{\oplus 15} \rightarrow
K \rightarrow
0
\end{equation}
and into the following chain of extensions
\begin{equation}
\begin{gathered}
\label{Kappa3--9}
0 \rightarrow
\mQ(-3) \rightarrow
K_1 \rightarrow
\Sym^2\mQ(-2) \rightarrow
0\\
0 \rightarrow
K_1 \rightarrow
K_2 \rightarrow
\Sym^3\mQ(-1) \rightarrow
0\\
0 \rightarrow
K_2 \rightarrow
K\rightarrow
\Sym^4\mQ \rightarrow
0.
\end{gathered}
\end{equation}
One can directly check using \eqref{of-40} and \eqref{Kappa3--9} that $H^0(K) \cong \Sym^4 V_3$ and $H^1(K) \cong V_3$.
The conclusion follows by considering the direct sum of \eqref{of-40} and \eqref{of0-1}.

\hypertarget{Fano3--10}{\subsection*{Fano 3--10}}
\subsubsection*{Mori-Mukai} Blow up of $\mathbb{Q}_3$ in the disjoint union of 2 conics.
\subsubsection*{Our description} $\mZ(\of(1,0,1) \oplus \of(0,1,1) \oplus \of(0,0,2)) \subset \PP^1 \times \PP^1 \times \PP^4$.
\subsubsection*{Identification} It suffices to apply twice Lemma \ref{lem:blowup}.

\hypertarget{Fano3--11}{\subsection*{Fano 3--11}}

\subsubsection*{Mori-Mukai} Blow up of \hyperlink{Fano2--35}{2--35} in an elliptic curve which is the intersection of two divisors from $|-\frac{1}{2}K|.$
\subsubsection*{Our description}  
$\mZ(\of(1,1,1) \oplus \mQ_{\PP^2}(0,0,1)) \subset \PP^1 \times \PP^2 \times \PP^3$.
\subsubsection*{Identification} We recall first that \hyperlink{Fano2--35}{2--35} is the blow up of $\PP^3$ at a point, which we have already identified as $\mZ(\mQ_{\PP^2}(0,1)) \subset \PP^2 \times \PP^3$. As such, its anticanonical class is $\of(2,2)$ by adjunction. It then suffices to apply Lemma \ref{lem:blowup}.

\hypertarget{Fano3--12}{\subsection*{Fano 3--12}}
\subsubsection*{Mori-Mukai} 	
Blow up of $\PP^3$ in the disjoint union of a line and a twisted cubic.
\subsubsection*{Our description} $\mZ(\of(0,1,1)\oplus \of(0,1,1) \oplus \of (1,0,1)) \subset \PP^1 \times \PP^2 \times \PP^3.$

\subsubsection*{Identification} The variety $\mZ(\of(1,1)\oplus \of(1,1) ) \subset \PP^2 \times \PP^3$ is the Fano 3-fold 2--27, the blow up of $\PP^3$ in a twisted cubic. The result then follows by Lemma \ref{lem:blowup}, with the two extra $(0,1)$ divisors cutting a line in space which by construction is disjoint from the twisted cubic. To make this explicit, take coordinates $[z_0,z_1], \  [y_0, y_1, y_2], \ [x_0,\ldots, x_3]$. The divisor of degree $(1,0,1)$ is therefore given by an expression of type $\sum z_i f_i(x_i)$. Say for simplicity $z_0 x_0 +z_1 x_3$. The line $L$ in $\PP^3$ which we are blowing up is therefore given by $x_0=x_3=0$. On the other hand the two divisors of degree $(0,1,1)$ define the twisted cubic as follows: they are given by the solutions of, e.g., \[  
 {\begin{pmatrix}
   x_0 & x_1 & x_2 \\
   x_1 & x_2 & x_3 \\
  \end{pmatrix} }
  {\begin{pmatrix} y_0 \\
  y_1 \\
  y_2
  \end{pmatrix}}=0.
  \]
In particular this locus is trivially identified with the blow up of $\PP^3$ where the matrix drops rank, that is $ {\rank \begin{pmatrix}
   x_0 & x_1 & x_2 \\
   x_1 & x_2 & x_3 \\
  \end{pmatrix} } <2$. The latter are the equations of the twisted cubic in $\PP^3$, which we can easily check to be disjoint from the line $L$.

\hypertarget{Fano3--14}{\subsection*{Fano 3--14}}
\subsubsection*{Mori-Mukai} 	
Blow up of $\PP^3$ in the disjoint union of a plane cubic curve and a point outside the plane.
\subsubsection*{Our description} $\mZ( \Lambda(0,1,0) \oplus \of(1,1,0) \oplus \mQ_{\PP^2}(1,0,0)) \subset \PP^3 \times \PP^{10} \times \PP^2$, where the bundle $\Lambda \in \Ext^1(\Sym^2 \mQ,\mQ(-1))$ is a uniquely defined extension on $\PP^3$ fitting into sequence \eqref{Lambda1--12} above.

\subsubsection*{Identification} The first two bundles on $\PP^3 \times \PP^{10}$ determine \hyperlink{Fano2--28}{2--28}, i.e., the blow up of $\PP^3$ in a plane cubic curve. To blow it up in a point, we can apply Lemma \ref{lem:blow} for the base $\PP^3$, adding a $\PP^2$ factor and the corresponding bundle. The extra point will in general be outside the plane.

\hypertarget{Fano3--15}{\subsection*{Fano 3--15}}
\subsubsection*{Mori-Mukai} 	
Blow up of $\mathbb{Q}_3$ in the disjoint union of a line and a conic.
\subsubsection*{Our description}  $\mZ(\of(1,0,1) \oplus \of(0,1,1) \oplus \mQ_{\PP^2}(0,0,1)) \subset \PP^1 \times \PP^2 \times \PP^4$.

\subsubsection*{Identification} By Lemma \ref{lem:blow} the zero locus of the last two bundles on $\PP^1 \times \PP^2 \times \PP^4$ gives us $\PP^1 \times \Bl_{\PP^1}\mathbb{Q}_3$. We still have to cut with a section of $\of(1,0,1)$. By Lemma \ref{lem:blowup} this is the blow up of $\Bl_{\PP^1} \mathbb{Q}_3$ in the locus cut by two linear sections, which is in general disjoint from the $\PP^1$. The result follows.

\hypertarget{Fano3--16}{\subsection*{Fano 3--16}} 
\subsubsection*{Mori-Mukai} Blow up of \hyperlink{Fano2--35}{2--35} in the proper transform of a twisted cubic containing the centre of the blow up.
\subsubsection*{Our description}
$\mZ(\of(0,1,1) \oplus \of(1,0,1) \oplus \mQ_{\PP^2_1}(0,1,0)) \subset \PP^2 \times \PP^3 \times \PP^2$.

\subsubsection*{Identification} We first fix the system of coordinates
$\PP^2_{[y_0\ldots y_2]} \times \PP^3_{[x_0\ldots x_3]} \times \PP^2_{[w_0\ldots w_2]}.$ As a first step we use Lemma \ref{lem:blow} to identify $\mQ_{\PP^2} (0,1) \subset \PP^2 \times \PP^3 $ as \hyperlink{Fano2--35}{2--35}, i.e., $\Bl_p \PP^3$. The two remaining divisors are, on $\PP^2 \times \PP^3$, of degree $(0,1)$ and $(1,0)$ and are both trivially identified with linear forms on $\PP^3$, but with a distinction. Without loss of generality, assume that $p$ is the point $[1,0,0,0]$. We have $(f \in |\of(1,0)|) \in \mathrm{Ann}(p)$, while $(g \in |\of(0,1)|)$ gives a non-zero element of $V_4^{\vee}/\mathrm{Ann}(p)$. In other words, $f=f(x_1,x_2,x_3)$ does not contain the coordinate $x_0$, while the converse holds for $g$. Both the divisors were twisted by $\of_{\PP^2}(1)$, giving rise to two divisors of degree $(1,1)$ on $\Bl_p\PP^3 \times \PP^2_{[w_0\ldots w_2]}$. As in \hyperlink{Fano3--12}{3--12}, these lead to the blow up of $\Bl_p\PP^3$ in a twisted cubic, that (since $f \in \mathrm{Ann}(p)$) passes through the point $p \in \PP^3$. The result follows.

\hypertarget{Fano3--18}{\subsection*{Fano 3--18}}
\subsubsection*{Mori-Mukai} Blow up of $\PP^3$ in the disjoint union of a line and a conic.
\subsubsection*{Our description} $\mZ(\mQ_2(0;0,0) \oplus \of(0;1,1) \oplus \of(1;0,1)) \subset \PP^1 \times \Fl(1,2,5)$.

\subsubsection*{Identification}
This Fano can be evidently identified with the blow up of \hyperlink{Fano2--30}{2--30} in a line disjoint from the conic. Recall that we described \hyperlink{Fano2--30}{2--30} as $(\mZ(\mQ_2 \oplus \of(1,1)) \subset \Fl(1,2,5)) \cong \mZ(\of(1)) \subset \PP_{\PP^3}(\of(-1)\oplus \of)$. The result then follows from Lemma \ref{lem:blowup}, since two divisors of degree $(0,1)$ cut a line in the base $\PP^3$.

We can write an alternative description for this Fano, based on the alternative description already given for \hyperlink{Fano2--30}{2--30}. Using Lemma \ref{lem:blowup} the Fano \hyperlink{Fano3--18}{3--18} will be

\[
\mZ(\of(1,1,0) \oplus \of(0,1,1) \oplus  \mQ_{\PP^3}(0,0,1)) \subset \PP^1\times \PP^3 \times \PP^4.
\]

\hypertarget{Fano3--19}{\subsection*{Fano 3--19}}
\subsubsection*{Mori-Mukai}Blow up of $\mathbb{Q}_3$ in two non-collinear points.
\subsubsection*{Our description}  $\mZ(\mQ_{\PP^2}(0,1) \oplus \of(0,2)) \subset \PP^2 \times \PP^4$.

\subsubsection*{Identification}
By Lemma \ref{lem:blow}, the first divisor yields the blow up of $\PP^4$ along a line. The second divisor is identified with a general quadric in $\PP^4$, hence it cuts out a quadric hypersurface in $\PP^4$ blown up along two points. The general quadric does not contain the line, so the blown up points are in general non-collinear.

\hypertarget{Fano3--20}{\subsection*{Fano 3--20}} 
\subsubsection*{Mori-Mukai} Blow up of $\mathbb{Q}_3$ in the disjoint union of two lines.
\subsubsection*{Our description}
$\mZ(\of(1,0,1) \oplus \mQ_{\PP^2_1}(0,1,0) \oplus \mQ_{\PP^2_2}(0,1,0) ) \subset \PP^2 \times \PP^4 \times \PP^2$.

\subsubsection*{Identification} We remark that, by Lemma \ref{lem:blow}, another model for \hyperlink{Fano2--31}{2--31} (the blow up of $\mathbb{Q}_3$ in one line) is given by $\mZ(\of(1,1) \oplus \mQ_{\PP^2} (0,1)) \subset \PP^2 \times \PP^4$. Our model for \hyperlink{Fano3--20}{3--20} is just the iteration of the blow up process, where the second and the third bundles give the blow up of $\PP^4$ along two disjoint lines $L_1, L_2$  and the first bundle gives a quadric which contains both the lines. Notice that in fact a section $\sum_k f_{1,k} f_{2,k}$ of the bundle $\of(1,0,1)$ identifies a quadric in $\PP^4$ and $f_{i,k} \in \Ann(L_i)$ for $i=1,2$ (see also the arguments used for \hyperlink{Fano2--19}{2--19} and \hyperlink{Fano3--16}{3--16}).

\hypertarget{Fano3--21}{\subsection*{Fano 3--21}}
\subsubsection*{Mori-Mukai} Blow up of $\PP^1 \times \PP^2$ in a curve of degree $(2,1)$.
\subsubsection*{Our description} $\mZ(\of(0,1,1) \oplus \Lambda(0,0,1)) \subset \PP^1 \times \PP^2 \times \PP^6$, being $\Lambda \in \Ext^1(\mQ_{\PP^2}^{\oplus 2},\of(1,-1))$ a uniquely defined extension on $\PP^1 \times \PP^2$ fitting into sequence \eqref{Lambda2-18}.

\subsubsection*{Identification}
On $\PP^1 \times \PP^2$, a general complete intersection of a $(0,1)$ and a $(1,2)$ divisors is a smooth curve of degree $(2,1)$. In order to blow it up, we can use Lemma \ref{lem:blowDegeneracyLocus}, according to which our Fano will be the zero locus of $\of(1) \otimes \pi^*(0,1)$ over the projective bundle $\pi:\PP(\of(-1,-1) \oplus \of)\rightarrow \PP^1 \times \PP^2$.

The above projective bundle has already been found when dealing with \hyperlink{Fano2--18}{2--18}: it is the zero locus of $\Lambda(0,0,1)$ over $\PP^1 \times \PP^2 \times \PP^6$, with $\Lambda$ fitting into \eqref{Lambda2-18}.

\hypertarget{Fano3--22}{\subsection*{Fano 3--22}}
\subsubsection*{Mori-Mukai} Blow up of $\PP^1 \times \PP^2$ in a conic on $\lbrace x \rbrace \times \PP^2 $, $\lbrace x \rbrace \in \PP^1$.
\subsubsection*{Our description} $\mZ(\of(1,0,1) \oplus \Lambda(0,0,1)) \subset \PP^1 \times \PP^2 \times \PP^6$, being $\Lambda \in \Ext^1(\Sym^2 \mQ,\mQ(-1))$ a uniquely defined extension on $\PP^2$ fitting into sequence \eqref{Lambda2-36}.
\subsubsection*{Identification}
We need to blow up on $\PP^1 \times \PP^2$ a complete intersection curve given by two divisors of degree $(1,0)$ and $(0,2)$. To do that, we use Lemma \ref{lem:blowDegeneracyLocus}: our Fano will then be the zero locus of $\of(1)$ over the projective bundle $\PP(\of(-1,0) \oplus \of(0,-2))$.

To find the above projective bundle, we can add the standard (pulled back) Euler sequence on $\PP^1$ to \eqref{inclusion2-36} and get
\[
0 \rightarrow
\of(-1,0) \oplus \of(0,-2) \rightarrow
\of^{\oplus 8} \rightarrow
\of(1,0) \oplus \Lambda \rightarrow
0,
\]
being $\Lambda$ a uniquely defined extension on $\PP^2$ fitting into sequence \eqref{Lambda2-36}. The conclusion follows.

\hypertarget{Fano3--23}{\subsection*{Fano 3--23}}
\subsubsection*{Mori-Mukai}  Blow up of \hyperlink{Fano2--35}{2--35} in the proper transform of a conic containing the centre of the blow up.
\subsubsection*{Our description} $\mZ(\mQ_{\PP^2}(0,1,0) \oplus \of(1,0,1) \oplus \mQ_{\PP^3}(0,0,1)) \subset \PP^2 \times \PP^3 \times \PP^4$.

\subsubsection*{Identification} By Lemma \ref{lem:blow} the first bundle (when seen on the first two factors) gives $X:=$ \hyperlink{Fano2--35}{2--35}, the blow up of $\PP^3$ in one point $p$. We need to blow up $X$ along the proper transform of a conic $Q$ containing $p$. Note that $Q$ is cut out by a hyperplane and a quadric in $\PP^3$ both containing $p$, so that $Q$ is the degeneracy locus of a map $\of_X(-1,-1) \oplus \of_X(-1,0) \rightarrow \of_X$ (see, e.g., the arguments used for \hyperlink{Fano2--19}{2--19} and \hyperlink{Fano3--16}{3--16}). Lemma \ref{lem:blowDegeneracyLocus} yields that our Fano will be the zero locus of $\of(1) \otimes \pi^*(\of(1,0))$ over the projective bundle $\pi:\PP(\of(0,-1) \oplus \of) \rightarrow \PP^2 \times \PP^3$.

Such projective bundle can be found in $\PP^2 \times \PP^3 \times \PP^4$, as the sequence on $\PP^2 \times \PP^3$
\[
0 \rightarrow
\of(0,-1) \oplus \of \rightarrow
\of^{\oplus 5} \rightarrow
\mQ_{\PP^3} \rightarrow
0
\]
shows. The conclusion follows.

\hypertarget{Fano3--24}{\subsection*{Fano 3--24}}
\subsubsection*{Mori-Mukai} The fiber product of 2--32 with $\Bl_p\PP^2$ over $\PP^2$.
\subsubsection*{Our description} $\mZ(\of(1,1,0) \oplus \of(0,1,1)) \subset \PP^1 \times \PP^2 \times \PP^2.$
\subsubsection*{Identification} See \cite[\textsection 77]{corti}.

\hypertarget{Fano3--25}{\subsection*{Fano 3--25}}
\subsubsection*{Mori-Mukai} $\PP_{\PP^1 \times \PP^1}(\of(0,-1) \oplus \of(-1,0))$, or the blow up of $\PP^3$ in two disjoint lines.
\subsubsection*{Our description} $\mZ(\of(0,1)^{\oplus 2}) \subset \Fl(1,2,4)$.

\subsubsection*{Identification} We can identify $\Fl(1,2,4)$ with $\PP_{\Gr(2,4)}(\mU)$. Let $Z:= \PP^1 \times \PP^1$. The two $(0,1)$ sections give us $\PP_Z(\mU|_Z)$. By \cite[Theorem 1.4]{ottaviani} the restriction of $\mU$ to $Z$ coincides with the direct sum of $\of(0,-1) \oplus \of(-1,0)$. The result follows.

An alternative description is $\mZ(\of(1,1,0) \oplus \of(1,0,1)) \subset \PP^3 \times \PP^1 \times \PP^1$, by simply apply twice Lemma \ref{lem:blowup}.

\hypertarget{Fano3--26}{\subsection*{Fano 3--26}}
\subsubsection*{Mori-Mukai} Blow up of $\PP^3$ in the disjoint union of a point and a line.
\subsubsection*{Our description}  $\mZ(\of(1,0,1) \oplus \mQ_{\PP^2}(0,0,1)) \subset \PP^1 \times \PP^2 \times \PP^3$.

\subsubsection*{Identification} The bundle $\mQ_{\PP^2} (0,1)$ on $ \PP^2 \times \PP^3$ gives the Fano \hyperlink{Fano2--35}{2--35} by Lemma \ref{lem:blow}. Two extra sections of $\of(0,1)$ on this space cut a line that does not intersect the exceptional divisor (equivalently, a line in $\PP^3$ that does not pass through the blown up point). The identification therefore follows by Lemma \ref{lem:blowup}.

\hypertarget{Fano3--28}{\subsection*{Fano 3--28}}
\subsubsection*{Mori-Mukai} $\PP^1 \times \Bl_p \PP^2$.
\subsubsection*{Our description} $\mZ(\of(1,0,1)) \subset \PP^1 \times \PP^1 \times \PP^2$.
\subsubsection*{Identification} See Lemma \ref{lem:blowup}.

\hypertarget{Fano3--29}{\subsection*{Fano 3--29}}
\subsubsection*{Mori-Mukai} Blow up of \hyperlink{Fano2--35}{2--35} in a line on the exceptional divisor.
\subsubsection*{Our description} $\mZ(\mQ_{\PP^2}(1,0,0) \oplus \mQ_{\PP^3}(0,0,1) \oplus \Lambda(0,0,1) \oplus \of(0,-1,1)) \subset \PP^3 \times \PP^2 \times \PP^9$, being $\Lambda \in \Ext^1(\Sym^2 \mQ,\mQ(-1))$ a uniquely defined extension on $\PP^2$ fitting into sequence \eqref{Lambda2-36}.
\subsubsection*{Identification} By Lemma \ref{lem:blow} the first bundle gives, on the first two factors, the blow up $Y$ of $\PP^3$ along a point. We then need to blow up a line in the exceptional divisor. By \cite[Corollary 9.12]{EisenbudHarris3264}, the exceptional divisor is a $(1,-1)$ divisor in $Y$; in order to cut out a line on it, we have to intersect it with the strict transform of a hyperplane in $\PP^3$ passing through the point, which is a $(0,1)$ divisor (see, e.g., the argument used for \hyperlink{Fano3--16}{3--16}).

Summarising, we need to blow $Y$ up along the intersection of the two aforementioned divisors. By Lemma \ref{lem:blowDegeneracyLocus}, this yields that our Fano variety is the zero locus of $\pi^* \of(0,-1) \otimes \of(1)$ over $\pi:\PP(\of(-1,0) \oplus \of(0,-2)) \rightarrow Y$. 

To express the above projective bundle, we can add the standard (pulled back) Euler sequence on $\PP^3$ to \eqref{inclusion2-36} and get
\[
0 \rightarrow
\of(-1,0) \oplus \of(0,-2) \rightarrow
\of^{\oplus 10} \rightarrow
\mQ_{\PP^3} \oplus \Lambda \rightarrow
0,
\]
being $\Lambda$ a uniquely defined extension on $\PP^2$ fitting into sequence \eqref{Lambda2-36}. The conclusion follows.

\begin{caveat}
\label{caveatBundle}
The above bundle $\of(0,-1,1)$ has clearly no sections on $\PP^3 \times \PP^2 \times \PP^9$, so our notation seems misleading. In fact, this bundle acquires a $4$-dimensional space of global sections once it is restricted to the zero locus of the previous ones, so that the direct sum above should be taken with a pinch of salt.

This phenomenon naturally occurs when we need to consider the exceptional divisor of a blow up obtained via Lemma \ref{lem:blow}: as already remarked, if $Y=\Bl_{\PP^{n-m-1}}\PP^n=\mZ(\mQ_{\PP^m}(0,1))\subset \PP^m \times \PP^n$, then the exceptional divisor is a $(-1,1)$ divisor in $Y$. Notice that $\of_{\PP^n \times \PP^m}(-1,1)|_Y \cong \of_Y(-1,1)$ indeed has global sections.
\end{caveat}

\hypertarget{Fano3--30}{\subsection*{Fano 3--30}}
\subsubsection*{Mori-Mukai} Blow up of \hyperlink{Fano2--35}{2--35} in the proper transform of a line containing the centre of the blow up.
\subsubsection*{Our description} $\mZ(\of(1,0,1) \oplus \mQ_{\PP^2}(0,1,0)) \subset \PP^2 \times \PP^3 \times \PP^1$.
\subsubsection*{Identification} By Lemma \ref{lem:blow} the second bundle (when seen on the first two factors) gives a Fano $X$ which is \hyperlink{Fano2--35}{2--35}, the blow up of $\PP^3$ in one point $p$. We need to blow up $X$ along the proper transform of a line containing $p$, which is the complete intersection of two divisors of degree $(1,0)$ on $\PP^2 \times \PP^3$ (see, e.g., the argument used for \hyperlink{Fano3--16}{3--16}). We conclude by Lemma \ref{lem:blowup}.

\hypertarget{Fano3--31}{\subsection*{Fano 3--31}}
\subsubsection*{Mori-Mukai} Blow up of the cone over a smooth quadric in $\PP^3$ in the vertex, or $\PP_{\PP^1 \times \PP^1}(\of(-1,-1) \oplus \of)$.
\subsubsection*{Our description} of $\mZ(\mQ_2 \oplus \of(0,2)) \subset \Fl(1,2,5)$.

\subsubsection*{Identification} By Corollary \ref{cor:blowupflag} we have that $\mZ(\mQ_2) \subset \Fl(1,2,5)$ is isomorphic to $\PP_{\PP^3}(\of(-1) \oplus \of)$. The extra quadric cuts only the base $\PP^3$, and yields the identification. 

We want to give an alternative description as \[\mZ(\mQ_{\PP^3}(0,1)  \oplus \of(2,0)) \subset \PP^3 \times \PP^4.\] 
By Lemma \ref{lem:blow}, $\mQ_{\PP^3} (0,1)$ gives the blow up of $\PP^4$ at a point $p_0$, with dual coordinate $x_0$. A section of $\of(2,0)$ gives a quadric in the space $\Sym^2(V_5^{\vee}/\langle x_0 \rangle)$. This gives the equation of a cone over a smooth, degenerate quadric in $\PP^3_{[x_1, \ldots, x_4]}$. The result follows.

\hypertarget{Fano4--2}{\subsection*{Fano 4--2}}
\subsubsection*{Mori-Mukai} 
Blow up of the cone over a smooth quadric in $\PP^3$ in the disjoint union of the vertex and an elliptic curve on the quadric. 
\subsubsection*{Our description} $\mZ(\mQ_{\PP^3}(0,1,0) \oplus \of(2,0,0) \oplus \of(0,1,1) \oplus \mQ_{\PP^4}(0,0,1)) \subset \PP^3 \times \PP^4 \times \PP^5$.

\subsubsection*{Identification} 
We use the alternative description of \hyperlink{Fano3--31}{3--31}. In fact to blow up the requested elliptic curve it suffices to blow up $Y:=$ \hyperlink{Fano3--31}{3--31}, in its intersection with a hyperplane not passing through the vertex of the cone and a general quadric, i.e., in the intersection of a $(0,1)$ and a $(0,2)$ divisors. Lemma \ref{lem:blowDegeneracyLocus} yields that our Fano variety is the zero locus of $\pi^* \of(0,1) \otimes \of(1)$ over $\pi:\PP(\of(0,-1)\oplus \of) \rightarrow Y$. Such projective bundle can be obtained as the zero locus of the remaining bundle by considering the Euler sequence on $\PP^4$, which yields an embedding of $\of(0,-1)\oplus \of$ inside $\PP(\of^{\oplus 5} \oplus \of)$.

\hypertarget{Fano4--3}{\subsection*{Fano 4--3}}
\subsubsection*{Mori-Mukai} Blow up of $\PP^1 \times \PP^1 \times \PP^1$ in a curve of degree $(1,1,2)$.
\subsubsection*{Our description} $\mZ(\of(1,1,0,1) \oplus \of (0,0,1,1)) \subset \PP^1 \times \PP^1 \times \PP^1 \times \PP^2.$

\subsubsection*{Identification} A complete intersection of divisors of degree $(1,1,0)$, $(1,1,1)$ is a curve of degree $(1,1,2)$ in $\PP^1 \times \PP^1 \times \PP^1$. In order to blow it up, we use Lemma \ref{lem:blowDegeneracyLocus}: our Fano is then the zero locus of $\of(1) \otimes \pi^*\of(1,1,0)$ over $\pi:\PP(\of(0,0,-1) \oplus \of)\rightarrow \PP^1 \times \PP^1 \times \PP^1$. From the standard Euler sequence on $\PP^1$ we get
\[
0 \rightarrow
\of(0,0,-1) \oplus \of \rightarrow
\of^{\oplus 3} \rightarrow
\of(0,0,1) \rightarrow
0,
\]
hence the conclusion.

\hypertarget{Fano4--4}{\subsection*{Fano 4--4}}
\subsubsection*{Mori-Mukai} Blow up of \hyperlink{Fano3--19}{3--19} in the proper transform of a conic through the points.
\subsubsection*{Our description} $\mZ(\mQ_{\PP^2}(0,1,0) \oplus \of(0,2,0) \oplus \of(1,0,1)) \subset \PP^2 \times \PP^4 \times \PP^1. $

\subsubsection*{Identification} The first two bundles on $\PP^2 \times \PP^4$ give the Fano \hyperlink{Fano3--19}{3--19}. We then just need to use Lemma \ref{lem:blowup}, since two sections of $\of(1,0)$ cut the three dimensional quadric in a conic passing through the two points (see also the argument used for \hyperlink{Fano3--16}{3--16}).

\hypertarget{Fano4--5}{\subsection*{Fano 4--5}}
\subsubsection*{Mori-Mukai} Blow up of $\PP^1 \times \PP^2$ in the disjoint union of a curve of degree $(2,1)$ and a curve of degree $(1,0)$.
\subsubsection*{Our description} $\mZ(\of(0,1,1,0) \oplus \Lambda(0,0,1,0) \oplus \of(0,1,0,1)) \subset \PP^1 \times \PP^2 \times \PP^6 \times \PP^1 $, where the bundle $\Lambda \in \Ext^1(\mQ_{\PP^2}^{\oplus 2},\of(1,-1))$ is a uniquely defined extension on $\PP^1 \times \PP^2$ fitting into sequence \eqref{Lambda2-18}.

\subsubsection*{Identification} The first two bundles describe, on $\PP^1 \times \PP^2 \times \PP^6$, the variety \hyperlink{Fano3--21}{3--21}. We need to blow it up along a curve of degree $(1,0)$, which is the complete intersection of two divisors of degree $(0,1)$ on $\PP^1 \times \PP^2$. The result follows from Lemma \ref{lem:blowup}.

\hypertarget{Fano4--6}{\subsection*{Fano 4--6}}
\subsubsection*{Mori-Mukai} Blow up of $\PP^3$ in the disjoint union of 3 lines.
\subsubsection*{Our description} $\mZ(\of(1,1,0,0) \oplus \of(1,0,1,0) \oplus \of(1,0,0,1)) \subset \PP^3 \times \PP^1 \times \PP^1 \times \PP^1. $

\subsubsection*{Identification} It suffices to apply three times Lemma \ref{lem:blowup}. By dimension reasons the three lines on $\PP^3$ which are cut each times are disjoint.

\hypertarget{Fano4--7}{\subsection*{Fano 4--7}}
\subsubsection*{Mori-Mukai} 	
Blow up of 2--32 in the disjoint union of a curve of degree $(0,1)$ and a curve of degree $(1,0)$.
\subsubsection*{Our description}  $\mZ(\of(1,0; 1,0) \oplus \of(0,1;0,1)) \subset \Fl(1,2,3) \times \PP^1 \times \PP^1. $

\subsubsection*{Identification} The flag variety $F:=\Fl(1,2,3)$ can be identified with 2--32, that is a $(1,1)$ section of $\PP^2 \times (\PP^2)^{\vee}.$ Notice that under this identification the generators of the Picard group of the flag are the restriction of the canonical ones on $\PP^2 \times (\PP^2)^{\vee}.$ In particular $H^0(F, \of_{F}(1,0)) \cong V_3^{\vee}$ and $H^0(F, \of_{F}(0,1)) \cong V_3$.
The zero locus of two sections of $\of_{F}(1,0)$ is a $(0,1)$ curve, and the opposite holds for $\of_{F}(0,1)$. We then apply twice Lemma \ref{lem:blowup} to conclude.

Of course thanks to the above identification and Lemma \ref{lem:blowup} this Fano can be described as well as 
\[
\mZ(\of(1,1,0,0) \oplus \of(1,0,1,0) \oplus \of(0,1,0,1)) \subset \PP^2 \times \PP^2 \times \PP^1 \times \PP^1.
\]

\hypertarget{Fano4--8}{\subsection*{Fano 4--8}}
\subsubsection*{Mori-Mukai} Blow up of \hyperlink{Fano3--31}{3--31} (i.e., $\PP_{\PP^1 \times \PP^1}(\of(-1,-1) \oplus \of)$) in a $(1,1)$-section of the base $\PP^1 \times \PP^1$, or blow up of $\PP^1 \times \PP^1 \times \PP^1$ in a curve of degree $(0,1,1)$.
\subsubsection*{Our description} $\mZ(\mQ_2 \oplus  \of(0,2;0)  \oplus \of(1,0;1) ) \subset \Fl(1,2,5) \times \PP^1.$

\subsubsection*{Identification} We use the first description by Mori--Mukai, together with our description of \hyperlink{Fano3--31}{3--31}.  Given this, it suffices to apply Lemma \ref{lem:blowup}, since the zero locus of two extra copies of $\of_{F}(1,0)$ on $\mZ(\mQ_2 \oplus \of_{F}(0,2)) \subset F:=\Fl(1,2,5)$ is such a curve. In fact $Z:=\mZ(\mQ_2 \oplus \of_{F}(0,2) \oplus \of_{F}(1,0)) \subset F$ corresponds to the base $\PP^1 \times \PP^1$; on $Z$, both the restrictions $\of_{F}(1,0)|_Z \cong \of_{F}(0,1)|_Z $ coincide with $\of_{\PP^1 \times \PP^1}(1,1)$, as can be easily checked via a Chern classes computation.

Alternatively, we can use Lemma \ref{lem:blowup} to give another description of this Fano, given the alternative one for \hyperlink{Fano3--31}{3--31}. In particular \hyperlink{Fano4--8}{4--8} will be given as
\[
\mZ(\mQ_{\PP^3}(0,1,0)  \oplus \of(2,0,0) \oplus \of(0,1,1)) \subset \PP^3 \times \PP^4 \times \PP^1.
\]
\hypertarget{Fano4--9}{\subsection*{Fano 4--9}}
\subsubsection*{Mori-Mukai} Blow up of \hyperlink{Fano3--25}{3--25} in an exceptional rational curve $E$ of the blow up.
\subsubsection*{Our description} $\mZ(\mQ_{\PP^2}(0,1,0,0) \oplus \of(1,0,1,0) \oplus \of(0,1,0,1)) \subset \PP^2 \times \PP^3 \times \PP^1 \times \PP^1.$

\subsubsection*{Identification} First we use that the bundle $\mQ_{\PP^2} (0,1) \subset \PP^2 \times \PP^3$ gives the blow up $\Bl_p \PP^3$ by Lemma \ref{lem:blow}. Lemma \ref{lem:blowup} yields that the other two bundles yield the blow up along two other lines $L, L'$ in $\PP^3$: $L$ (corresponding to $\of(1,0,1,0)$) passing through $p$, $L'$ avoiding it (see, e.g., the argument used for \hyperlink{Fano3--16}{3--16}). Therefore we identify the above variety with $\Bl_{\Sigma} \PP^3$, where $\Sigma:= L \cup L' \cup p$, and $p \in L$. This is the same as $\Bl_{E}(\Bl_{L \cup L'} \PP^3)$. Since the exceptional divisor of the second blow up $\pi_2$ is a $\PP^1$-bundle over the union of the two lines, (with $E=\pi_2^{-1}(p)$) the result follows.

\hypertarget{Fano4--10}{\subsection*{Fano 4--10}}
\subsubsection*{Mori-Mukai} $\PP^1 \times \Bl_2 \PP^2$.
\subsubsection*{Our description} $\mZ(\of(0,1,1) \oplus \mQ_{\PP^2}(0,0,1)) \subset \PP^1 \times \PP^2 \times \PP^3$.

\subsubsection*{Identification} Lemma \ref{lem:blow} identifies the zero locus of a general section of the second bundle with $\PP^1 \times \Bl_p \PP^3$. A section of the remaining bundle gives a quadric in $\PP^3$ containing $p$ (see, e.g., the argument used for \hyperlink{Fano2--19}{2--19}), which identifies our model with $\Bl_p (\PP^1 \times \PP^1)$, which is isomorphic to the blow up of $\PP^2$ in two points. We remark that Lemma \ref{lem:blowup} provides another simple model, i.e., the zero locus of $\of(1,0,1,0)\oplus \of(0,1,1,0)$ over $\PP^1 \times \PP^1 \times \PP^2 \times \PP^1$.

\hypertarget{Fano4--11}{\subsection*{Fano 4--11}}
\subsubsection*{Mori-Mukai} Blow up of \hyperlink{Fano3--28}{3--28} in $\lbrace x\rbrace \times E$, $x \in \PP^1$ and $E$ the $(-1)$-curve.
\subsubsection*{Our description}. $\mZ(\of(0,1,1,0) \oplus \mQ_{\PP^2}(0,0,0,1) \oplus \Lambda(0,0,0,1) \oplus \of(0,0,-1,1)) \subset \PP^1 \times \PP^2 \times \PP^1 \times \PP^6$, being $\Lambda \in \Ext^1(\of(0,1)^{\oplus 2}, \of(1, -1))$ a uniquely defined extension on $\PP^1 \times \PP^1$ fitting into \eqref{Lambda3-2}.
\subsubsection*{Identification}
By Lemma \ref{lem:blow} the first bundle defines, on $\PP^1 \times \PP^2 \times \PP^1$, the Fano \hyperlink{Fano3--28}{3--28}. By \cite[Corollary 9.12]{EisenbudHarris3264}, we need to blow up the intersection of a $(1,0,0)$ and a $(0,1,-1)$ divisors. Using Lemma \ref{lem:blowDegeneracyLocus}, our Fano will be the zero locus of $\of(1) \otimes \pi^*\of(0,0,-1)$ over the projective bundle $\pi:\PP(\of(-1,0,-1) \otimes \of(0,-1,0)) \rightarrow \PP^1 \times \PP^2 \times \PP^1$.

For $\of(-1,0,-1)$ we can pull back sequence \eqref{inclusion3-2} and get
\begin{equation}
0 \rightarrow \of(-1,0,-1) \rightarrow \of^{\oplus 4} \rightarrow \Lambda \rightarrow 0,
\end{equation}
where $\Lambda$ fits into \eqref{Lambda3-2}. Adding it with the standard Euler sequence on $\PP^2$, we get
\[
0 \rightarrow \of(-1,0,-1) \oplus \of(0,-1,0) \rightarrow \of^{\oplus 7} \rightarrow \Lambda \oplus \mQ_{\PP^2} \rightarrow 0,
\]
which gives the conclusion.

We remark that the last bundle in the description should be taken with a caveat, as it has no global sections on the ambient space, but acquires some when restricted to the zero locus of the previous bundles. See Caveat \ref{caveatBundle}.

\hypertarget{Fano4--12}{\subsection*{Fano 4--12}}
\subsubsection*{Mori-Mukai} Blow up of \hyperlink{Fano2--33}{2--33} in the disjoint union of two exceptional lines of the blow up.
\subsubsection*{Our description} $\mZ(\of(1,1,0) \oplus \Lambda(0,0,1) \oplus \of(-1,1,1)) \subset \PP^1 \times \PP^3 \times \PP^{8}$, being $\Lambda \in \Ext^1(\mQ_{\PP^3}^{\oplus 2},\of(1,-1))$ a uniquely defined extension on $\PP^1 \times \PP^3$ fitting into sequence \eqref{Lambda4-12} below.
\subsubsection*{Identification} By Lemma \ref{lem:blow} (or Lemma \ref{lem:blowup}) the first bundle gives, on the first two factors, the blow up $Y$ of $\PP^3$ along a line. We then need to blow up two disjoint lines in the exceptional divisor. By \cite[Corollary 9.12]{EisenbudHarris3264}, the exceptional divisor is a $(-1,1)$ divisor in $Y$; in order to cut out two lines on it, we have to intersect it with the strict transform of a general quadric hypersurface in $\PP^3$, which is a $(0,2)$ divisor cutting the blown up line in two points.

Summarising, we need to blow $Y$ up along the intersection of the two aforementioned divisors. By Lemma \ref{lem:blowDegeneracyLocus}, this yields that our Fano variety is the zero locus of $\pi^* \of(-1,1) \otimes \of(1)$ over the projective bundle $\pi:\PP(\of(-1,-1) \oplus \of) \rightarrow \PP^1 \times \PP^3$.

To describe this projective bundle, we can argue as in Lemma \ref{projBundle1-12} or Lemma \ref{projBundle2-2}: we combine the (pull back of the) two (possibly twisted) Euler sequences
\begin{gather*}
0 \rightarrow \of(-1,-1) \rightarrow \of(0,-1)^{\oplus 2} \rightarrow \of(1,-1) \rightarrow 0,\\
0 \rightarrow \of(0,-1)^{\oplus 2} \rightarrow \of^{\oplus 8} \rightarrow \mQ_{\PP^3}^{\oplus 2} \rightarrow 0.
\end{gather*}
We get
\begin{gather}
\label{inclusion4-12}
0 \rightarrow \of(-1,-1) \rightarrow \of^{\oplus 8} \rightarrow \Lambda \rightarrow 0,
\\
\label{Lambda4-12}
0 \rightarrow \of(1,-1) \rightarrow \Lambda \rightarrow \mQ_{\PP^3}^{\oplus 2} \rightarrow 0,
\end{gather}
where the rank $7$ bundle $\Lambda$ is homogeneous, not completely reducible and globally generated, and its space of global sections coincides with $H^0(\PP^3, 
\mQ^{\oplus 2}) \cong (V_4)^{\oplus 2}$. Adding $\of \rightarrow \of$ to \eqref{inclusion4-12} we get that $\PP(\of(-1,-1) \oplus \of)$ is the zero locus of $\Lambda(0,0,1)$ in $\PP^1 \times \PP^3 \times \PP^8$, whence the conclusion.

We remark that the last bundle in the description should be taken with a caveat, as it has no global sections on the ambient space, but acquires some when restricted to the zero locus of the previous bundles. See Caveat \ref{caveatBundle}.

\hypertarget{Fano4--13}{\subsection*{Fano 4--13}}
\subsubsection*{Mori-Mukai} Blow up of $\PP^1 \times \PP^1 \times \PP^1$ in a curve of degree $(1,3,1)$.
\subsubsection*{Our description}
$\mZ(\Lambda(0,0,0,1) \oplus \of(1,0,1,1)) \subset \PP^1 \times \PP^1 \times \PP^1 \times \PP^4$, being $\Lambda \in \Ext^1(\of(0,1)^{\oplus 2}, \of(1, -1))$ a uniquely defined extension on $\PP^1 \times \PP^1$ (the first two copies) fitting into \eqref{Lambda3-2}.
\subsubsection*{Identification}
The complete intersection between a $(2,1,1)$ and a $(1,0,1)$ divisors is a curve of degree $(1,3,1)$ in $\PP^1 \times \PP^1 \times \PP^1$. In order to blow it up, we use Lemma \ref{lem:blowDegeneracyLocus}: our Fano $Y$ will be the zero locus of $\of(1) \otimes \pi^*\of(1,0,1)$ over $\pi:\PP(\of(-1,-1,0) \oplus \of)\rightarrow \PP^1 \times \PP^1 \times \PP^1$. From \eqref{inclusion3-2} we get that this projective bundle is the zero locus of $\Lambda(0,0,0,1)$ over $\PP^1 \times \PP^1 \times \PP^1 \times \PP^4$, where $\Lambda$ is a bundle on $\PP^1 \times \PP^1$ (the first two copies) fitting into \eqref{Lambda3-2}. The conclusion follows.

Analogously, we could have used the complete intersection of a $(3,1,0)$ and a $(1,0,1)$ divisors, which is again a curve of degree $(1,3,1)$. A similar argument requires the projective bundle $\PP(\of(-2,-1,0) \oplus \of(0,0,-1))$ and produces a model $Y'$ in $\PP^1 \times \PP^1 \times \PP^1 \times \PP^7$. If we consider the normal sequence for $Y=\mZ(\mathcal{F}) \subset \PP:=\PP^1 \times \PP^1 \times \PP^1 \times \PP^4$,
a few cohomology computations via the Koszul complex as described in Section \ref{computeinvariants} provide that $h^0(T_{\PP}|_Y)=33, h^0(\mathcal{F}|_Y)=34$ and the higher cohomology groups vanish. In \cite[Lemma 8.11]{pcs} it is shown that the family of curves of degree $(1,1,3)$ on $(\PP^1)^3$ has dimension one (up to the action of $\Aut((\PP^1)^3)$), and that for all but one curve the automorphism group is finite. This means that a general model $Y$ admits a $(34-33=1)$-dimensional family of deformations, which is the dimension of the moduli of Fano \hyperlink{Fano4--13}{4--13}, hence $Y$ is general in moduli. The corresponding computations for $Y'\subset \PP^1 \times \PP^1 \times \PP^1 \times \PP^8$ give, analogously, $73-72=1$, so that the models $Y'$ are also general in the moduli space of Fano \hyperlink{Fano4--13}{4--13}.

\hypertarget{Fano5--1}{\subsection*{Fano 5--1}}
\subsubsection*{Mori-Mukai} Blow up of \hyperlink{Fano2--29}{2--29} in the disjoint union of three exceptional lines of the blow up.
\subsubsection*{Our description}
$\mZ(\of(1,1,0,0) \oplus \Lambda(0,0,1,0) \oplus \of(-1,1,1,0) \oplus \mQ_{\PP^3}(0,0,0,1) \oplus \mQ_{\PP^8}(0,0,0,1)) \subset \PP^1 \times \PP^3 \times \PP^{8} \times \PP^{11}$, being $\Lambda \in \Ext^1(\mQ_{\PP^3}^{\oplus 2},\of(1,-1))$ a uniquely defined extension on $\PP^1 \times \PP^3$ fitting into sequence \eqref{Lambda4-12}.

\subsubsection*{Identification} This Fano variety is the blow up of \hyperlink{Fano4--12}{4--12} along a rational curve, as per the alternative description given in \cite[Table 5]{morimukai}. If we consider the model for \hyperlink{Fano4--12}{4--12} in $\PP^1 \times \PP^3 \times \PP^8$ (given by the zero locus of the first three bundles), we can check that the intersection of a $(0,1,0)$ divisor and a $(0,0,1)$ divisor is indeed a rational curve $C$, and the corresponding blow up $Y$ can be checked to have the right Hodge diamond and invariants. To ensure that $Y$ is Fano (hence, it is \hyperlink{Fano5--1}{5--1}) we can check that a fiber $F$ of the exceptional divisor has $F.K_Y=-1$, so that $-K_Y$ is ample by \cite[Thm 1.4.3]{isp5}.

As usual, we blow up $C$ via Lemma \ref{lem:blowDegeneracyLocus}: our Fano will be the zero locus of $\of(1)$ over $\PP(\of(0,-1,0) \oplus \of(0,0,-1))$. This projective bundle can be easily described by considering the direct sum of the two Euler sequences, which yields
\[
0 \rightarrow \of(0,-1,0) \oplus \of(0,0,-1) \rightarrow \of^{\oplus 13} \rightarrow \mQ_{\PP^3} \oplus \mQ_{\PP^8} \rightarrow 0.
\]
The conclusion follows.

\hypertarget{Fano5--2}{\subsection*{Fano 5--2}}
\subsubsection*{Mori-Mukai} Blow up of \hyperlink{Fano3--25}{3--25} in the disjoint union of two exceptional lines on the same irreducible component.

\subsubsection*{Our description}  $\mZ(\of(1,1,0,0) \oplus \Lambda(0,0,1,0) \oplus \of(-1,1,1,0) \oplus \of(0,1,0,1)) \subset \PP^1 \times \PP^3 \times \PP^{8} \times \PP^1$, being $\Lambda \in \Ext^1(\mQ_{\PP^3}^{\oplus 2},\of(1,-1))$ a uniquely defined extension on $\PP^1 \times \PP^3$ fitting into sequence \eqref{Lambda4-12}.

\subsubsection*{Identification} 
Recall that \hyperlink{Fano4--12}{4--12} is the blow up of $\PP^3$ in a line and then in the disjoint union of two exceptional lines of the blow up, and is given by the zero locus of the first three bundles. To get \hyperlink{Fano5--2}{5--2} we need to blow it up along the strict transform of a line not intersecting any of the other three. The previously found model for \hyperlink{Fano4--12}{4--12} was in $\PP^1 \times \PP^3 \times \PP^{8}$, and such a line is the complete intersection of two $(0,1,0)$ divisors. Lemma \ref{lem:blowup} yields the conclusion.

\section{Tables}
\label{tables}
In this last section we collect in an exhaustive table all the models for Fano 3-folds we exhibited in Section \ref{Fano3folds}, together with the models already existing in the literature. In Table \ref{tab:3folds}, MM stands for the Mori--Mukai numeration; the Picard rank $\rho$ is the first number. In the column ``Inv'' an entry $(a,b,c)$ means the invariants $(h^0(-K), K^3, h^{2,1})$ of the corresponding Fano. The column $X$ refers to the ambient variety, whereas $\mathcal{F}$ is the bundle whose zero locus produces the 3-fold. In some cases alternative descriptions (marked by ``\emph{alt.}'') are given, whenever we find them equally interesting. In the column ``Notes'' we put either the reference for the chosen model, when it was not provided by us, or a further explanation of the bundles appearing in the previous column.

We include a second table, Table \ref{tab:delpezzo}, for Del Pezzo surfaces, whose models can be easily figured out from Table \ref{tab:3folds}. Each family in the table (except 2--1) will correspond to the blow up of $\PP^2$ in $9-K^2$ points in sufficiently general position. All models (for 3-folds and Del Pezzo surfaces) are general.

\begin{centering}
\begin{scriptsize}
\setlength\tabcolsep{4pt}
\begin{longtable}{ccccc}

\caption{Fano 3-folds.}\label{tab:3folds}\\
\toprule
MM&
Inv& 
$X$ & 
$\mathcal{F}$ & 
Notes \\
\cmidrule(lr){1-1}\cmidrule(lr){2-2}\cmidrule(lr){3-3} \cmidrule(lr){4-4} \cmidrule(lr){5-5}
\endfirsthead
\multicolumn{5}{l}{\vspace{-0.25em}\scriptsize\emph{\tablename\ \thetable{} continued from previous page}}\\
\toprule
MM&
Inv& 
$X$ & 
$\mathcal{F}$ & 
Notes \\
\cmidrule(lr){1-1}\cmidrule(lr){2-2}\cmidrule(lr){3-3} \cmidrule(lr){4-4} \cmidrule(lr){5-5}
\endhead
\multicolumn{5}{r}{\scriptsize\emph{Continued on next page}}\\
\endfoot
\bottomrule
\endlastfoot

\hyperlink{Fano1--1}{1--1} & $(4,2,52)$ &\ $\PP(1^4,3)$\ & $\of(6)$&\cite{isp5}\\
\multicolumn{2}{l}{\rule{10pt}{0pt}\emph{alt.}}& $\PP^3 \times \PP^{20}$&$\of(0,2) \oplus K(0,1)$&$K \in \Ext^2_{\PP^3}(\Sym^3 \mQ, \mQ(-2))$\\
\evnrow 1--2 & $(5,4,30)$& $\PP^4$& $\of(4)$& \cite{isp5} \\
1--3 & $(6,6,20)$& $\PP^5$& $\of(2) \oplus \of(3)$& \cite{isp5} \\
\evnrow 1--4 & $(7,8,14)$& $\PP^6$& $\of(2)^{\oplus 3}$&\cite{isp5} \\
1--5 & $(8,10,10)$& $\Gr(2,5)$& $\of(2) \oplus \of(1)^{\oplus 2}$&\cite{isp5} \\
\evnrow 1--6 & $(9,12,7)$& $\OGr^+(5,10)$& $\of(\frac{1}{2})^{\oplus 7}$&\cite{isp5} \\
\evnrow \multicolumn{2}{l}{\rule{10pt}{0pt}\emph{alt.}} & $\Gr(2,5)$& $\mU^{\vee}(1)\oplus \of(1)$&\cite{corti} \\
1--7 & $(10,14,5)$& $\Gr(2,6)$&
$\of(1)^{\oplus 5}$&\cite{isp5} \\
\evnrow 1--8 & $(11,16,3)$& $\Gr(3,6)$& $\W^2\mU^{\vee} \oplus \of(1)^{\oplus 3}$&\cite{isp5} \\
1--9 & $(12,18,2)$& $\Gr(2,7)$& $\mQ^{\vee}(1) \oplus \of(1)^{\oplus 2}$&\cite{isp5} \\
\evnrow 1--10 & $(14,22,0)$& $\Gr(3,7)$& $(\W^2\mU^{\vee})^{\oplus 3}$&\cite{isp5} \\
1--11 & $(7,8,21)$& $\PP(1^3,2,3)$& $\of(6)$&\cite{isp5} \\
\evnrow \hyperlink{Fano1--12}{1--12} & $(11,16,10)$& $\PP(1^4,2)$& $\of(4)$&\cite{isp5} \\
\evnrow \multicolumn{2}{l}{\rule{10pt}{0pt}\emph{alt.}} & $\PP^3 \times \PP^{10}$& $\Lambda(0,1) \oplus \of(0,2)$&$\Lambda \in \Ext^1_{\PP^3}(\Sym^2 \mQ, \mQ(-1))$ \\
1--13 & $(15,24,5)$& $\PP^4$& $\of(3)$&\cite{isp5} \\
\evnrow 1--14 & $(19,32,2)$& $\PP^5$& $\of(2)^{\oplus 2}$&\cite{isp5} \\
1--15 & $(23,40,0)$& $\Gr(2,5)$& $\of(1)^{\oplus 3}$&\cite{isp5} \\
\evnrow 1--16 & $(30,54,0)$& $\PP^4$& $\of(2)$&\cite{isp5} \\
1--17 & $(35,64,0)$& $\PP^3$& & \cite{isp5} \\
\evnrow 2--1 & $(5,4,22)$& $\PP(1^3,2,3) \times \PP^1$& $\of(6,0) \oplus \of(1,1)$&\cite{corti} \\
\hyperlink{Fano2--2}{2--2} & $(6,6,20)$& $\PP^1\times \PP^2 \times \PP^{12}$& $\of(0,0,2) \oplus K(0,0,1)$&$K \in \Ext_{\PP^1 \times \PP^2}^2(\of(1,0)^{\oplus 6}, \mQ_{\PP^2}(-1,-1))$ \\
\evnrow \hyperlink{Fano2--3}{2--3} & $(7,8,11)$& $\PP(1^4,2) \times \PP^1$& $\of(4,0) \oplus \of(1,1)$&\cite{corti} \\
\evnrow \multicolumn{2}{l}{\rule{10pt}{0pt}\emph{alt.}} & $\PP^3 \times \PP^{10} \times \PP^1$& $\Lambda(0,1,0) \oplus \of(0,2,0) \oplus \of(1,0,1)$&$\Lambda \in \Ext^1_{\PP^3}(\Sym^2 \mQ, \mQ(-1))$ \\
2--4 & $(8,10,10)$& $\PP^1 \times \PP^3$& $\of(1,3)$&\cite{corti} \\
\evnrow \hyperlink{Fano2--5}{2--5} & $(9,12,6)$& $\PP^1 \times \PP^4$& $\of(0,3) \oplus \of(1,1)$& \\
2--6 & $(9,12,9)$& $\PP^2 \times \PP^2$& $\of(2,2)$&\cite{corti} \\
\evnrow 2--7 & $(10,14,5)$& $\PP^1 \times \PP^4$& $\of(0,2) \oplus \of(1,2)$&\cite{corti} \\
\hyperlink{Fano2--8}{2--8} & $(10,14,9)$& $\PP^2 \times \PP^3 \times \PP^{12}$& $\Lambda(0,0,1) \oplus \of(0,0,2)$&$\Lambda \in \Ext^1_{\PP^2 \times \PP^3}(\mQ_{\PP^3}^{\oplus 3}, \mQ_{\PP^2}(0, -1))$ \\
\evnrow 2--9 & $(11,16,5)$& $\PP^2 \times \PP^3$& $\of(1,1) \oplus \of(1,2)$& \cite{corti} \\
\hyperlink{Fano2--10}{2--10} & $(11,16,3)$& $\Gr(2,4) \times \PP^1$& $\of(2,0) \oplus \of(1,1)$& \\
\evnrow \hyperlink{Fano2--11}{2--11} & $(12,18,5)$& $\PP^2 \times \PP^4$& $\mQ_{\PP^2}(0,1) \oplus \of(1,2)$& \\
\evnrow \multicolumn{2}{l}{\rule{10pt}{0pt}\emph{alt.}} & $\Fl(1,3,5)$& $\mQ_2^{\oplus 2} \oplus \of(2,1)$& \\
2--12 & $(13,20,3)$& $\PP^3 \times \PP^3$& $\of(1,1)^{\oplus 3}$&\cite{corti} \\
\evnrow 2--13 & $(13,20,2)$& $\PP^2 \times \PP^4$& $\of(1,1)^{\oplus 2} \oplus \of(0,2)$&\cite{corti} \\
2--14 & $(13,20,1)$& $\Gr(2,5) \times \PP^1$& $\of(1,0)^{\oplus 3} \oplus \of(1,1)$&\cite{corti} \\
\evnrow \hyperlink{Fano2--15}{2--15} & $(14,22,4)$& $\PP^3 \times \PP^4$& $\mQ_{\PP^3}(0,1) \oplus \of(2,1)$& \\
\evnrow \multicolumn{2}{l}{\rule{10pt}{0pt}\emph{alt.}} & $\Fl(1,2,5)$& $\mQ_2 \oplus \of(1,2)$& \\
 \hyperlink{Fano2--16}{2--16} & $(14,22,2)$& $\PP^2 \times \Gr(2,4)$& $\mU^{\vee}_{\Gr(2,4)}(1,0) \oplus \of(0,2)$& \\
  \multicolumn{2}{l}{\rule{10pt}{0pt}\emph{alt.}} & $\Fl(1,2,4)$& $\of(1,0) \oplus \of(0,2)$& \\
  \evnrow \hyperlink{Fano2--17}{2--17} & $(15,24,1)$& $\Gr(2,4) \times \PP^3$& $\mU_{\Gr(2,4)}^{\vee}(0,1) \oplus \of(1,1) \oplus \of(1,0)$&\cite{corti} \\
  \evnrow \multicolumn{2}{l}{\rule{10pt}{0pt}\emph{alt.}} & $\Fl(1,2,4)$& $\of(0,1) \oplus \of(1,1)$& \\
   \hyperlink{Fano2--18}{2--18} & $(15,24,2)$& $\PP^1 \times \PP^2 \times \PP^6$& $\Lambda(0,0,1) \oplus \of(0,0,2)$&$\Lambda \in \Ext^1_{\PP^1 \times \PP^2}(\mQ_{\PP^2}^{\oplus 2}, \of(1,-1))$ \\
   \evnrow \hyperlink{Fano2--19}{2--19} & $(16,26,2)$& $\PP^3 \times \PP^5$& $\mQ_{\PP^3}(0,1) \oplus \of(1,1)^{\oplus 2}$& \\
   \evnrow \multicolumn{2}{l}{\rule{10pt}{0pt}\emph{alt.}} & $\Fl(1,3,6)$& $\mQ_2^{\oplus 2} \oplus \of(1,1)^{\oplus 2}$& \\
  2--20 & $(16,26,0)$& $\Gr(2,5) \times \PP^2$& $\mU^{\vee}_{\Gr(2,5)}(0,1) \oplus \of(1,0)^{\oplus 3}$&\cite{corti} \\
  \evnrow 2--21 & $(17,28,0)$& $\Gr(2,4) \times \PP^4$& $\mU^{\vee}_{\Gr(2,4)}(0,1)^{\oplus 2} \oplus \of(1,0)$&\cite{corti} \\
 \hyperlink{Fano2--22}{2--22} & $(18,30,0)$& $\PP^3 \times \Gr(2,5)$& $\mQ_{\Gr(2,5)}(1,0) \oplus \of(0,1)^{\oplus 3}$& \\
  \multicolumn{2}{l}{\rule{10pt}{0pt}\emph{alt.}} & $\Fl(1,2,5)$& $\of(1,0) \oplus \of(0,1)^{\oplus 3}$&\cite{corti}  \\
  \evnrow \hyperlink{Fano2--23}{2--23} & $(18,30,1)$& $\PP^4 \times \PP^5$& $\mQ_{\PP^4}(0,1) \oplus \of(2,0) \oplus \of(1,1)$& \\*
  \evnrow \multicolumn{2}{l}{\rule{10pt}{0pt}\emph{alt.}} & $\Fl(1,2,6)$& $\mQ_2 \oplus \of(0,2) \oplus \of(1,1)$& \\
  2--24 & $(18,30,0)$& $\PP^2 \times \PP^2$& $\of(1,2)$& \cite{isp5} \\ 
  \evnrow 2--25 & $(19,32,1)$& $\PP^1 \times \PP^3$& $\of(1,2)$&\cite{corti} \\
  \hyperlink{Fano2--26}{2--26} & $(20,34,0)$& $\Gr(2,4) \times \Gr(2,5)$& $\mQ_{\Gr(2,4)} \boxtimes \mU_{\Gr(2,5)}^{\vee} \oplus \of(1,0) \oplus \of(0,1)^{\oplus 2}$& \\
   \multicolumn{2}{l}{\rule{10pt}{0pt}\emph{alt.}} & $\Fl(2,3,5)$& $\mU_1^{\vee} \oplus \of(1,0) \oplus \of(0,1)^{\oplus 2}$& \\
  \evnrow 2--27 & $(22,38,0)$& $\PP^3 \times \PP^2$& $\of(1,1)^{\oplus 2}$&\cite{corti} \\
   \hyperlink{Fano2--28}{2--28} & $(23,40,1)$& $\PP^3 \times \PP^{10}$& $\Lambda(0,1) \oplus \of(1,1)$&$\Lambda \in \Ext^1_{\PP^3}(\Sym^2\mQ, \mQ(-1))$ \\
   \evnrow \hyperlink{Fano2--29}{2--29} & $(23,40,0)$& $\PP^1 \times \PP^4$& $\of(0,2) \oplus \of(1,1)$& \\
   \hyperlink{Fano2--30}{2--30} & $(26,46,0)$& $\PP^3 \times \PP^4$& $\mQ_{\PP^3}(0,1) \oplus \of(1,1)$& \\
   \multicolumn{2}{l}{\rule{10pt}{0pt}\emph{alt.}} & $\Fl(1,2,5)$& $\mQ_2 \oplus \of(1,1)$& \\
   \evnrow \hyperlink{Fano2--31}{2--31} & $(26,46,0)$& $\PP^2 \times \Gr(2,4)$& $\mU^{\vee}_{\Gr(2,4)}(1,0) \oplus \of(0,1)$& \\
  2--32 & $(27,48,0)$& $\PP^2 \times \PP^2$& $\of(1,1)$&\cite{isp5} \\
 \multicolumn{2}{l}{\rule{10pt}{0pt}\emph{alt.}}& $\Fl(1,2,3)$& &\cite{isp5} \\
 \evnrow \hyperlink{Fano2--33}{2--33} & $(30,54,0)$& $\PP^1 \times \PP^3$& $\of(1,1)$& \\
 2--34 & $(30,54,0)$& $\PP^1 \times \PP^2$& &\cite{isp5} \\
 \evnrow \hyperlink{Fano2--35}{2--35} & $(31,56,0)$& $\PP^2 \times \PP^3$& $\mQ_{\PP^2}(0,1)$& \\
 \evnrow \multicolumn{2}{l}{\rule{10pt}{0pt}\emph{alt.}} & $\Fl(1,2,4)$& $\mQ_2$& \\
 \hyperlink{Fano2--36}{2--36} & $(34,62,0)$& $\PP^2 \times \PP^6$& $\Lambda(0,1)$&$\Lambda \in \Ext^1_{\PP^2}(\Sym^2\mQ, \mQ(-1))$ \\

 \evnrow \hyperlink{Fano3--1}{3--1} & $(9,12,8)$& $\PP^1 \times \PP^1 \times \PP^1 \times \PP^8$& $K(0,0,0,1) \oplus \of(0,0,0,2)$&$K \in \Ext^2_{(\PP^1)^3}(\of(0,0,1)^{\oplus 4},\of(1,-1,-1))$ \\
  \hyperlink{Fano3--2}{3--2} & $(10,14,3)$& $\PP^1 \times \PP^1 \times \PP^5$& $\Lambda(0,0,1) \oplus \of(0,1,2)$&$\Lambda \in \Ext^1_{(\PP^1)^2}(\of(0,1)^{\oplus 2}, \of(1, -1))$ \\
  \evnrow 3--3 & $(12,18,3)$& $\PP^1 \times \PP^1 \times \PP^2$& $\of(1,1,2)$&\cite{corti} \\
  \hyperlink{Fano3--4}{3--4} & $(12,18,2)$& $\PP^1 \times \PP^2 \times \PP^6 \times \PP^1$& $\Lambda(0,0,1,0) \oplus \of(0,0,2,0) \oplus \of(0,1,0,1)$&$\Lambda \in \Ext^1_{\PP^1 \times \PP^2}(\mQ_{\PP^2}^{\oplus 2}, \of(1,-1))$ \\
  \evnrow \hyperlink{Fano3--5}{3--5} & $(13,20,0)$& $\PP^1 \times \PP^2 \times \PP^7$& $\Lambda(0,0,1) \oplus \of(0,1,1)^{\oplus 2}$&$\Lambda \in \Ext^1_{\PP^1 \times \PP^2}(\mQ_{\PP^2}^{\oplus 2}, \of(1,-1))$ \\
  \hyperlink{Fano3--6}{3--6} & $(14,22,1)$& $\PP^1 \times \PP^1 \times \PP^3$& $\of(1,0,2) \oplus \of(0,1,1)$& \\
  \evnrow 3--7 & $(15,24,1)$& $\PP^1 \times \PP^2 \times \PP^2$& $\of(0,1,1) \oplus \of(1,1,1)$&\cite{corti} \\
  \hyperlink{Fano3--8}{3--8} & $(15,24,0)$& $\PP^1 \times \PP^2 \times \PP^2$& $\of(0,1,2) \oplus \of(1,1,0)$& \\
  \evnrow  & & & &\multicolumn{1}{c}{$\Lambda \in \Ext^1_{\PP^2}(\Sym^2\mQ, \mQ(-1)),$} \\
  \evnrow \multirow{-2}{*}{\hyperlink{Fano3--9}{3--9}}& \multirow{-2}{*}{$(16,26,3)$} & \multirow{-2}{*}{$\PP^2 \times \PP^6 \times \PP^{20}$} & \multirow{-2}{*}{$\Lambda(0,1,0) \oplus \mQ_{\PP^6}(0,0,1) \oplus K(0,0,1)$}& \multicolumn{1}{c}{$K \in \Ext^3_{\PP^2}(\Sym^4\mQ, \mQ(-3))$}\\
  \hyperlink{Fano3--10}{3--10} & $(16,26,0)$& $\PP^1 \times \PP^1 \times \PP^4$& $\of(1,0,1) \oplus \of(0,1,1) \oplus \of(0,0,2)$& \\
  \evnrow \hyperlink{Fano3--11}{3--11} & $(17,28,1)$& $\PP^1 \times \PP^2 \times \PP^3$& $\of(1,1,1) \oplus \mQ_{\PP^2}(0,0,1)$& \\
  \hyperlink{Fano3--12}{3--12} & $(17,28,0)$& $\PP^1 \times \PP^2 \times \PP^3$& $\of(0,1,1) \oplus \of(0,1,1) \oplus \of(1,0,1)$& \\
  \evnrow 3--13 & $(18,30,0)$& $(\PP^2)^3$& $\of(1,1,0) \oplus \of(1,0,1) \oplus \of(0,1,1)$&\cite{corti} \\
  \hyperlink{Fano3--14}{3--14} & $(19,32,1)$& $\PP^3 \times \PP^{10} \times \PP^2$& $\Lambda(0,1,0) \oplus \mQ_{\PP^2}(1,0,0) \oplus \of(1,1,0)$&$\Lambda \in \Ext^1_{\PP^3}(\Sym^2\mQ, \mQ(-1)) $ \\
  \evnrow \hyperlink{Fano3--15}{3--15} & $(19,32,0)$& $\PP^1 \times \PP^2 \times \PP^4$& $\of(1,0,1) \oplus \of(0,1,1) \oplus \mQ_{\PP^2}(0,0,1)$& \\
  \hyperlink{Fano3--16}{3--16} & $(20,34,0)$& $\PP^2_1 \times \PP^2_2 \times \PP^3$& $\of(0,1,1) \oplus \of(1,1,0) \oplus \mQ_{\PP^2_1}(0,0,1)$& \\
  \evnrow 3--17 & $(21,36,0)$& $\PP^1 \times \PP^1 \times \PP^2$& $\of(1,1,1)$&\cite{corti} \\
  \hyperlink{Fano3--18}{3--18} & $(21,36,0)$& $\PP^1\times \PP^3 \times \PP^4$& $\of(1,1,0) \oplus \of(0,1,1) \oplus  \mQ_{\PP^3}(0,0,1)$& \\
   \multicolumn{2}{l}{\rule{10pt}{0pt}\emph{alt.}} & $\PP^1 \times \Fl(1,2,5)$& $\mQ_2(0;0,0) \oplus \of(0;1,1) \oplus \of(1;0,1)$& \\
   \evnrow \hyperlink{Fano3--19}{3--19} & $(22,38,0)$& $\PP^2 \times \PP^4$& $\of(0,2) \oplus \mQ_{\PP^2}(0,1)$& \\
   \hyperlink{Fano3--20}{3--20} & $(22,38,0)$& $\PP^2_1 \times \PP^2_2 \times \PP^4$& $\of(1,1,0) \oplus \mQ_{\PP^2_1}(0,0,1) \oplus \mQ_{\PP^2_2}(0,0,1)$& \\
   \evnrow \hyperlink{Fano3--21}{3--21} & $(22,38,0)$& $
   \PP^1 \times \PP^2 \times \PP^6$& $\of(0,1,1) \oplus \Lambda(0,0,1)$&$\Lambda \in \Ext^1_{\PP^1 \times \PP^2}(\mQ_{\PP^2}^{\oplus 2}, \of(1,-1)) $ \\
   \hyperlink{Fano3--22}{3--22} & $(23,40,0)$& $\PP^1 \times \PP^2 \times \PP^6$& $\of(1,0,1) \oplus \Lambda(0,0,1)$&$\Lambda \in \Ext^1_{\PP^2}(\Sym^2 \mQ, \mQ(-1))$ \\
 \evnrow \hyperlink{Fano3--23}{3--23} & $(24,42,0)$& $\PP^2 \times \PP^3 \times \PP^4$& $\mQ_{\PP^2}(0,1,0) \oplus \of(1,0,1) \oplus \mQ_{\PP^3}(0,0,1)$& \\
 \hyperlink{Fano3--24}{3--24} & $(24,42,0)$& $\PP^1 \times \PP^2 \times \PP^2$& $\of(1,1,0) \oplus \of(0,1,1)$&\cite{corti} \\
\evnrow \hyperlink{Fano3--25}{3--25} & $(25,44,0)$& $\PP^1 \times \PP^1 \times \PP^3$& $\of(1,0,1) \oplus \of(0,1,1)$& \\
\evnrow \multicolumn{2}{l}{\rule{10pt}{0pt}\emph{alt.}} & $\Fl(1,2,4)$& $\of(0,1)^{\oplus 2}$& \\
\hyperlink{Fano3--26}{3--26} & $(26,46,0)$& $\PP^1 \times \PP^2 \times \PP^3$& $\of(1,0,1) \oplus \mQ_{\PP^2}(0,0,1)$& \\
\evnrow 3--27 & $(27,48,0)$& $\PP^1 \times \PP^1 \times \PP^1$& &\cite{isp5} \\
\hyperlink{Fano3--28}{3--28} & $(27,48,0)$& $\PP^1 \times \PP^1 \times \PP^2$& $\of(1,0,1)$& \\
\evnrow & & & \multicolumn{1}{c}{$\mQ_{\PP^2}(1,0,0) \oplus \mQ_{\PP^3}(0,0,1) \oplus$} &
\\
\evnrow \multirow{-2}{*}{\hyperlink{Fano3--29}{3--29}} & \multirow{-2}{*}{$(28,50,0)$}& \multirow{-2}{*}{$\PP^3 \times \PP^2 \times \PP^9$}& \multicolumn{1}{c}{$\oplus \Lambda(0,0,1) \oplus  \of(0,-1,1)$}&\multirow{-2}{*}{$\Lambda \in \Ext^1_{\PP^2}(\Sym^2\mQ, \mQ(-1))$} \\
\hyperlink{Fano3--30}{3--30} & $(28,50,0)$& $\PP^1 \times \PP^2 \times \PP^3$& $\of(1,1,0) \oplus \mQ_{\PP^2}(0,0,1)$& \\
\evnrow \hyperlink{Fano3--31}{3--31} & $(29,52,0)$& $\PP^3 \times \PP^4$& $\mQ_{\PP^3}(0,1) \oplus \of(2,0)$& \\
\evnrow \multicolumn{2}{l}{\rule{10pt}{0pt}\emph{alt.}} & $\Fl(1,2,5)$& $\mQ_2 \oplus \of(0,2)$& \\
4--1 & $(15,24,1)$& $(\PP^1)^4$& $\of(1,1,1,1)$&\cite{isp5} \\
\evnrow  & &  & $\mQ_{\PP^3}(0,1,0) \oplus \mQ_{\PP^4}(0,0,1) \oplus $ &  \\
\evnrow \multirow{-2}{*}{\hyperlink{Fano4--2}{4--2}} & \multirow{-2}{*}{$(17,28,1)$}& \multirow{-2}{*}{$\PP^3 \times \PP^4 \times \PP^5$}& $\oplus \of(2,0,0) \oplus \of(0,1,1) $ &\\

 \hyperlink{Fano4--3}{4--3} & $(18,30,0)$& $(\PP^1)^3 \times \PP^2$& $\of(1,1,0,1) \oplus \of(0,0,1,1)$& \\
 
 \evnrow \hyperlink{Fano4--4}{4--4} & $(19,32,0)$& $\PP^1 \times \PP^2 \times \PP^4$& $\of(1,1,0) \oplus \of(0,0,2) \oplus \mQ_{\PP^2}(0,0,1)$& \\
 \hyperlink{Fano4--5}{4--5} & $(19,32,0)$& $\PP^1 \times \PP^2 \times \PP^6 \times \PP^1$& $\of(0,1,1,0) \oplus \Lambda(0,0,1,0) \oplus \of(0,1,0,1)$&$\Lambda \in \Ext^1_{\PP^1 \times \PP^2}(\mQ_{\PP^2}^{\oplus 2},\of(1,-1))$ \\
 \evnrow \hyperlink{Fano4--6}{4--6} & $(20,34,0)$& $(\PP^1)^3 \times \PP^3$& $\of(1,0,0,1) \oplus \of(0,1,0,1) \oplus \of(0,0,1,1)$& \\
 \hyperlink{Fano4--7}{4--7} & $(21,36,0)$& $(\PP^1)^2 \times (\PP^2)^2$& $\of(0,0,1,1) \oplus \of(1,0,1,0) \oplus \of(0,1,0,1)$&\\
 \multicolumn{2}{l}{\rule{10pt}{0pt}\emph{alt.}} & $ (\PP^1)^2  \times \Fl(1,2,3)$& $\of(1,0; 1,0) \oplus \of(0,1;0,1)$& \\
 \evnrow \hyperlink{Fano4--8}{4--8} & $(22,38,0)$& $\PP^1 \times \PP^3 \times \PP^4$& $\of(1,0,1) \oplus \of(0,2,0) \oplus \mQ_{\PP^3}(0,0,1)$& \\
 \evnrow \multicolumn{2}{l}{\rule{10pt}{0pt}\emph{alt.}} & $\PP^1 \times \Fl(1,2,5)$& $\of(1;1,0) \oplus \of(0;0,2) \oplus \mQ_2$& \\
 \hyperlink{Fano4--9}{4--9} & $(23,40,0)$& $(\PP^1)^2 \times \PP^2 \times \PP^3$& $\of(1,0,1,0) \oplus \of(0,1,0,1) \oplus \mQ_{\PP^2}(0,0,0,1)$& \\
 \evnrow \hyperlink{Fano4--10}{4--10} & $(24,42,0)$& $(\PP^1)^3 \times \PP^2$& $\of(1,0,0,1) \oplus \of(0,1,0,1)$& \\
 \evnrow \multicolumn{2}{l}{\rule{10pt}{0pt}\emph{alt.}} & $\PP^1\times \PP^2 \times \PP^3$& $\of(0,1,1) \oplus \mQ_{\PP^2}(0,0,1)$& \\

  & & & \multicolumn{1}{c}{$\of(0,1,1,0) \oplus \mQ_{\PP^2}(0,0,0,1) \oplus$}\\
  \multirow{-2}{*}{\hyperlink{Fano4--11}{4--11}}
  & \multirow{-2}{*}{$(25,44,0)$}& \multirow{-2}{*}{$\PP^1 \times \PP^2 \times \PP^1 \times \PP^6$}& \multicolumn{1}{c}{$\oplus \Lambda(0,0,0,1) \oplus \of(0,0,-1,1)$}&\multirow{-2}{*}{$\Lambda \in \Ext^1_{(\PP^1)^2}(\of(0,1)^{\oplus 2}, \of(1, -1))$} \\
  
  \evnrow \hyperlink{Fano4--12}{4--12} & $(26,46,0)$& $ \PP^1 \times \PP^3 \times \PP^{8}$& $\of(1,1,0) \oplus \Lambda(0,0,1) \oplus \of(-1,1,1)$&$\Lambda \in \Ext^1_{\PP^1 \times \PP^3}(\mQ_{\PP^3}^{\oplus 2},\of(1,-1))$ \\ 
  \hyperlink{Fano4--13}{4--13} & $(16,26,0)$& $ \PP^1_1 \times \PP^1_2 \times \PP^1_3 \times \PP^4$& $\Lambda(0,0,0,1) \oplus \of(1,0,1,1) $&$\Lambda \in \Ext^1_{\PP^1_1 \times \PP^1_2}(\of(0,1)^{\oplus 2}, \of(1, -1))$ \\
  \evnrow  & & & $\of(1,1,0,0) \oplus \Lambda(0,0,1,0) \oplus \of(-1,1,1,0) \oplus $& \\
  \evnrow \multirow{-2}{*}{\hyperlink{Fano5--1}{5--1}} & \multirow{-2}{*}{$(17,28,0)$}& \multirow{-2}{*}{$\PP^1 \times \PP^3 \times \PP^{8} \times \PP^{11}$}& $\oplus \mQ_{\PP^3}(0,0,0,1) \oplus \mQ_{\PP^8}(0,0,0,1)$& \multirow{-2}{*}{$\Lambda \in \Ext^1_{\PP^1 \times \PP^3}(\mQ_{\PP^3}^{\oplus 2},\of(1,-1))$}\\

& & & $\of(1,1,0,0) \oplus \Lambda(0,0,1,0) \oplus$
\\
\multirow{-2}{*}{\hyperlink{Fano5--2}{5--2}} & \multirow{-2}{*}{$(21,36,0)$}& \multirow{-2}{*}{$\PP^1 \times \PP^3 \times \PP^{8} \times \PP^1$}& $\oplus \of(-1,1,1,0) \oplus \of(0,1,0,1)$&\multirow{-2}{*}{$\Lambda \in \Ext^1_{\PP^1 \times \PP^3}(\mQ_{\PP^3}^{\oplus 2},\of(1,-1))$} \\
\evnrow 5--3 & $(21,36,0)$& $\PP^2 \times \PP^2 \times \PP^1$& $\of(1,1,0)^{\oplus 2}$& \cite{isp5} \\
\evnrow  \multicolumn{2}{l}{\rule{10pt}{0pt}\emph{alt.}} & $(\PP^1)^4$& $\of(1,1,1,0)$& \cite{isp5} \\
6--1 & $(18,30,0)$& $\Gr(2,5) \times \PP^1$& $\of(1,0)^{\oplus 4}$&\cite{isp5} \\
\evnrow 7--1 & $(15,24,0)$& $\PP^4 \times \PP^1$& $\of(2,0)^{\oplus 2}$&\cite{isp5} \\
8--1 & $(12,18,0)$& $\PP^3 \times \PP^1$& $\of(3,0)$&\cite{isp5} \\
\evnrow 9--1 & $(9,12,0)$& $\PP(1^3,2) \times \PP^1$& $\of(4,0)$&\cite{isp5} \\
\evnrow \multicolumn{2}{l}{\rule{10pt}{0pt}\emph{alt.}} & $\PP^1 \times \PP^2 \times \PP^1$& $\of(2,2,0)$&\cite{isp5} \\
10--1 & $(6,10,0)$& $\PP(1^2,2,3) \times \PP^1$& $\of(6,0)$&\cite{isp5} \\
\end{longtable}
\end{scriptsize}
\end{centering}

\begin{longtable}{cccccrc}
\caption{Del Pezzo surfaces.}\label{tab:delpezzo}\\
\toprule
\multicolumn{1}{c}{DP}&\multicolumn{1}{c}{$K^2$}& \multicolumn{1}{c}{X} & \multicolumn{1}{c}{$\mathcal{F}$} \\
\cmidrule(lr){1-1}\cmidrule(lr){2-2}\cmidrule(lr){3-3} \cmidrule(lr){4-4} 
\endfirsthead
\multicolumn{5}{l}{\vspace{-0.25em}\scriptsize\emph{\tablename\ \thetable{} continued from previous page}}\\
\midrule
\endhead
\multicolumn{5}{r}{\scriptsize\emph{Continued on next page}}\\
\endfoot
\bottomrule
\endlastfoot
\evnrow 1--1 & $9$& $\PP^2$&  \\
 2--1 & $8$& $\PP^1 \times \PP^1$&  \\
\evnrow 2--2 & $8$& $\PP^1 \times \PP^2$& $\of(1,1) $\\

3--1 & $7$& $\PP^1 \times \PP^1 \times \PP^2$& $\of(1,0,1) \oplus \of(0,1,1)$\\
\evnrow 4--1 & $6$& $\PP^2 \times \PP^2$& $\of(1,1)^{\oplus 2}$\\
\evnrow  & & $(\PP^1)^3$& $\of(1,1,1)$\\
5--1 & $5$& $\Gr(2,5)$& $\of(1)^{\oplus 4}$\\
\evnrow 6--1 & $4$& $\PP^4 $& $\of(2)^{\oplus 2}$\\
7--1 & $3$& $\PP^3 $& $\of(3)$\\
\evnrow 8--1 & $2$& $\PP(1^3,2)$& $\of(4)$\\
\evnrow & & $\PP^1 \times \PP^2$& $\of(2,2)$\\
9--1 & $1$& $\PP(1^2,2,3)$& $\of(6)$\\

\end{longtable}

\frenchspacing


\newcommand{\etalchar}[1]{$^{#1}$}

\end{document}